\documentclass[reqno]{amsart}
\usepackage{amsmath, amsfonts, amsthm, amssymb, bbm}
\usepackage{geometry}
\usepackage{lineno,hyperref}

\newtheorem{theorem}{Theorem}[section]
\newtheorem{lemma}[theorem]{Lemma}

\newtheorem{corollary}[theorem]{Corollary}
\newtheorem{proposition}[theorem]{Proposition}

\theoremstyle{definition}
\newtheorem{definition}[theorem]{Definition}

\theoremstyle{remark}

\numberwithin{equation}{section}

\newcommand{\spt}[1]{\mbox{\normalfont spt}\Parans{#1}}
\newcommand{\sptBar}[2]{\overline{\mbox{\normalfont spt}}_{#1}\Parans{#2}}
\newcommand{\Mspt}[1]{\mbox{\normalfont M2spt}\Parans{#1}}
\newcommand{\sptA}[2]{\mbox{\normalfont spt}_{#1}\Parans{#2}}

\newcommand{\Parans}[1]{\left(#1\right)}
\newcommand{\SBrackets}[1]{\left[#1\right]}
\newcommand{\PieceTwo}[4]
{
	\left\{
   	\begin{array}{ll}
      	#1 & #3 \\
       	#2 & #4
     	\end{array}
	\right.
}
\newcommand{\aqprod}[3]{\Parans{#1;#2}_{#3}}
\newcommand{\jacprod}[2]{\SBrackets{#1;#2}_{\infty}}
\newcommand{\Jac}[2]{\left(\frac{#1}{#2}\right)}

\newcommand{\STwoB}{\mbox{\rm S2}}
\newcommand{\SB}{\overline{\mbox{\rm S}}}

\allowdisplaybreaks
\begin{document}


\author{FRANK G. GARVAN*}
\thanks{*The first author was supported in part by a grant from the
Simon's Foundation (\#318714)}
\address{Department of Mathematics, University of Florida\\
Gainesville, Florida 32611, USA
\endgraf  fgarvan@ufl.edu}

\author{CHRIS JENNINGS-SHAFFER}
\address{Department of Mathematics, Oregon State University,
Corvallis, Oregon 97331, USA
\endgraf jennichr@math.oregonstate.edu}

\title{Exotic Bailey-Slater SPT-Functions II: Hecke-Rogers-Type Double Sums and 
Bailey Pairs From Groups A, C, E}



\begin{abstract}
We investigate spt-crank-type functions arising from Bailey pairs.
We recall four spt-type functions corresponding to the Bailey pairs
$A1$, $A3$, $A5$, and $A7$ of Slater and given four new spt-type functions
corresponding to Bailey pairs
$C1$, $C5$, $E2$, and $E4$.
Each of these functions can be thought of as a count on the number
of appearances of the smallest part in certain integer partitions.
We prove simple Ramanujan type congruences for these functions 
that are explained by a spt-crank-type function. 
The spt-crank-type functions are two variable $q$-series determined
by a Bailey pair, that when $z=1$ reduce to the spt-type functions.
We find the spt-crank-type functions to have interesting representations
as either infinite products or as Hecke-Rogers-type double series. 
These series reduce nicely when $z$ is a certain 
root of unity
and allow us to deduce the congruences. 
Additionally we find dissections when $z$ is a certain root of unity to
give another proof of the congruences. 
Our double sum and product formulas require Bailey's Lemma and conjugate Bailey pairs.
Our dissection formulas follow from Bailey's Lemma and dissections of known ranks and cranks. 
\end{abstract}
\allowdisplaybreaks

\maketitle

\section{Introduction}

A partition of $n$ is a non-increasing sequence of positive integers that sum to
$n$. For example, the partitions of $4$ are $4$, $3+1$, $2+2$, $2+1+1$ and
$1+1+1+1$. We have a 
weighted count on partitions given by counting a partition
by the number of times the smallest part appears. We let $\spt{n}$ denote this
weighted count of the partitions of $n$. From the partitions of $4$ we see that
$\spt{4}=10$. Andrews introduced the spt function in \cite{Andrews} and there
proved that $\spt{5n+4}\equiv 0\pmod{5}$, $\spt{7n+5}\equiv 0\pmod{7}$, and
$\spt{13n+6}\equiv 0\pmod{13}$.

In this article we use the standard product notation,
\begin{align*}
	&\aqprod{z}{q}{n} = \prod_{j=0}^{n-1} (1-zq^j),
	\hspace{20pt}
	\aqprod{z}{q}{\infty} = \prod_{j=0}^\infty (1-zq^j),
	\hspace{20pt}
	\jacprod{z}{q} = \aqprod{z,q/z}{q}{\infty},	
	\\
	&\aqprod{z_1,\dots,z_k}{q}{n} = \aqprod{z_1}{q}{n}\dots\aqprod{z_k}{q}{n},
	\hspace{20pt}
	\aqprod{z_1,\dots,z_k}{q}{\infty} = \aqprod{z_1}{q}{\infty}\dots\aqprod{z_k}{q}{\infty},
	\\
	&\jacprod{z_1,\dots,z_k}{q} = \jacprod{z_1}{q}\dots\jacprod{z_k}{q}
.
\end{align*}
Noting 
\begin{align*}
	\frac{q^n}{(1-q^n)^2} 
	&= 
	q^n + 2q^{2n}+ 3q^{3n}+ 4q^{4n}\dots 
	,
\end{align*}
we see a generating function for $\spt{n}$ is
given by
\begin{align*}
	S(q) 
	&= 
	\sum_{n=1}^\infty \spt{n}q^n
	=\sum_{n=1}^\infty \frac{q^{n}}{(1-q^n)^2\aqprod{q^{n+1}}{q}{\infty}}
.
\end{align*}

In \cite{AGL} Andrews, the first author, and Liang 
defined a two variable generalization of the spt function by
\begin{align*}
	S(z,q) 
	&=
	\sum_{n=1}^\infty\sum_{m=-\infty}^\infty N_S(m,n)z^mq^n
	= 
	\sum_{n=1}^\infty \frac{q^{n}\aqprod{q^{n+1}}{q}{\infty}}
		{\aqprod{zq^n,z^{-1}q^n}{q}{\infty}}
	= 
	\frac{\aqprod{q}{q}{\infty}}{\aqprod{z,z^{-1}}{q}{\infty}}
	\sum_{n=1}^\infty \frac{\aqprod{z,z^{-1}}{q}{n}q^{n}}
		{\aqprod{q}{q}{n}}
,
\end{align*}
so that $S(1,q)=S(q)$. There they reproved the congruences 
$\spt{5n+4}\equiv 0\pmod{5}$ and $\spt{7n+5}\equiv 0\pmod{7}$ by examining
$S(\zeta_5,q)$ and $S(\zeta_7,q)$, where $\zeta_5$ is a primitive fifth root
of unity and $\zeta_7$ is a primitive seventh root of unity. 
The idea of working with a two variable generalization of a generating
function to prove congruences began with Dyson's
rank of a partition \cite{AS, Dyson}
and continued with the Andrews-Garvan crank of a partition
\cite{AndrewsGarvan, Garvan1}. By introducing
an extra variable, we obtain a statistic and refinement of
the original counting function. An illuminating account of ranks
and cranks can be found in Chapters 2, 3, and 4 of \cite{AndrewsBerndt2}.
Essential to the study of $S(z,q)$ is the identity
\begin{align}\label{EqIntro1}
	(1-z)(1-z^{-1})S(z,q) &= R(z,q) - C(z,q)
,
\end{align}
where $R(z,q)$ is the generating function of the rank of a partition and
$C(z,q)$ is the generating function of the crank of a partition. 
To prove (\ref{EqIntro1}) one applies a limiting case
of Bailey's Lemma to a certain Bailey pair.

In \cite{GarvanJennings} the authors used Bailey's Lemma on four different 
Bailey pairs to obtain new spt-crank functions and prove congruences for three spt functions for 
overpartitions and the spt function for partitions with smallest part even and 
without repeated odd parts. An overpartition is a partition in which the first 
occurrence of a part may be overlined, for example the overpartitions of $3$
are $3$, $\overline{3}$, $2+1$, $2+\overline{1}$, $\overline{2}+1$, 
$\overline{2}+\overline{1}$, $1+1+1$, and $\overline{1}+1+1$. We let
$\sptBar{}{n}$ denote the total number of appearances of the smallest parts among
the overpartitions of $n$ whose smallest part is not overlined, for example
$\sptBar{}{3} = 6$. The other two spt functions for overpartitions are given
by additionally requiring the smallest part to be odd or requiring the smallest part
to be even. We let $\Mspt{n}$ denote the number of smallest parts in the partitions
of $n$ without repeated odd parts and smallest part even. For example the relevant
partitions of $6$ are $6$, $4+2$, and $2+2+2$ and so $\Mspt{n}=5$.
In all cases we found the spt-cranks, the two variable 
generalizations of the spt functions, to be the difference of a known rank
function and some sort of infinite product crank function. Furthermore this fact came directly from a 
Bailey pair. In particular, the generating functions for the respective two 
variable generalizations for $\sptBar{}{n}$ and $\Mspt{n}$ are
\begin{align*}
	\SB(z,q)
	&=
		\sum_{n=1}^\infty \frac{q^n\aqprod{-q^{n+1},q^{n+1}}{q}\infty}
			{\aqprod{zq^n,z^{-1}q^n}{q}{\infty}}
		=
		\frac{\aqprod{-q,q}{q}{\infty}}{\aqprod{z,z^{-1}}{q}{\infty}}
		\sum_{n=1}^\infty \frac{\aqprod{z,z^{-1}}{q}{n}q^n}
			{\aqprod{-q,q}{q}{n}}
	,\\
	\STwoB(z,q)
	&=
		\sum_{n=1}^\infty \frac{q^{2n}\aqprod{-q^{2n+1},q^{2n+2}}{q^2}\infty}
			{\aqprod{zq^{2n},z^{-1}q^{2n}}{q^2}{\infty}}
		=
		\frac{\aqprod{-q,q^2}{q^2}{\infty}}{\aqprod{z,z^{-1}}{q^2}{\infty}}
		\sum_{n=1}^\infty \frac{\aqprod{z,z^{-1}}{q^2}{n}q^{2n}}
			{\aqprod{-q,q^2}{q^2}{n}}
	,\\
	\SB(1,q) &= \sum_{n=1}^\infty \sptBar{}{n}q^n
	,\\
	\STwoB(1,q) &= \sum_{n=1}^\infty \Mspt{n}q^n.
\end{align*}
We let $\overline{R}(z,q)$ denote the generating function for the Dyson rank of 
an overpartition \cite{Lovejoy1},
 and let $R2(z,q)$ denote the $M_2$-rank of a partition without 
repeated odd parts \cite{BG2}. The Dyson rank of an overpartition is the largest part minus
the number of parts. The $M_2$-rank of a partition without repeated odd parts is
the ceiling of half the largest part minus the number of parts. We found that
\begin{align*}
	(1-z)(1-z^{-1})\SB(z,q)
	&= 
		\overline{R}(z,q) - \aqprod{-q}{q}{\infty}C(z,q)
	,\\
	(1-z)(1-z^{-1})\STwoB(z,q)
	&= 
		R2(z,q) - \aqprod{-q}{q^2}{\infty}C(z,q^2)
.
\end{align*}
Both of these identities were deduced from the same limiting case of Bailey's 
Lemma as (\ref{EqIntro1}), but applied to a different Bailey pair.

With this in mind we now consider a 
spt-crank-type function to be a function of the form
\begin{align*}
	\frac{P(q)}{\aqprod{z,z^{-1}}{q}{\infty}}\sum_{n=1}^\infty \aqprod{z,z^{-1}}{q}{n}q^n\beta_n
,
\end{align*}
where $P(q)$ is an infinite product and $\beta$ comes from a Bailey pair 
relative to $(1,q)$. We consider a spt-type function to be the $z=1$ case of a
spt-crank-type function.
We recall that a pair of sequences $(\alpha,\beta)$ is a Bailey pair relative
to $(a,q)$ if
\begin{align*}
	\beta_n = \sum_{k=0}^n \frac{\alpha_k}{\aqprod{q}{q}{n-k}\aqprod{aq}{q}{n+k}}
.
\end{align*}

The authors are in the process of studying all such spt-type functions that arise 
from the Bailey pairs of Slater from \cite{Slater} and \cite{Slater2}. 
In \cite{JS} the second 
author introduced the spt functions corresponding
to the pairs $A(1)$, $A(3)$, $A(5)$, and $A(7)$
and proved congruences for these functions by dissecting the spt-crank-type functions
when $z$ is a certain root of unity. In this paper we again consider the Bailey
pairs from group A as well as the Bailey pairs $C(1)$, $C(5)$, $E(2)$, and
$E(4)$. These Bailey pairs were selected because they lead to spt-type functions
with simple congruences and the spt-crank-type functions satisfy a variety of identities.
In a coming paper we handle Bailey pairs from groups $B$, $F$, $G$, and
$J$.

Our study of these spt-type functions can be described as follows.
We take a Bailey pair $(\alpha^X,\beta^X)$,
where each $\beta^X_n$ has integer coefficients as a series in $q$,
and define the corresponding two
variable spt-crank-type function by
\begin{align*}
	S_X(z,q) &= 
	\frac{P_X(q)}{\aqprod{z,z^{-1}}{q}{\infty}}
	\sum_{n=1}^\infty \aqprod{z,z^{-1}}{q}{n}\beta^X_nq^n	
	=
	\sum_{n=1}^\infty\sum_{m=-\infty}^\infty M_X(m,n)z^mq^n
.
\end{align*}
Here $P_X(q)$ is an infinite product of our choice.
By setting $z=1$ we get the spt-type function
\begin{align*}
	S_X(q) 
	&=
	S_X(1,q)
 	=
	\sum_{n=1}^\infty\left(\sum_{m=-\infty}^\infty M_X(m,n)\right)q^n
	=
	\sum_{n=1}^\infty \sptA{X}{n}q^n
.
\end{align*}
We prove simple linear congruences for $\sptA{X}{n}$ by considering
$S_{X}(\zeta,q)$ where $\zeta$ is a primitive root of unity.
For $t$ a positive integer we define
\begin{align*}
	M_X(k,t,n) &= \sum_{m\equiv k \pmod{t}}M_X(m,n)
.
\end{align*}
We note that
\begin{align*}
	\sptA{X}{n} &= \sum_{k=0}^{t-1} M_X(k,t,n),
\end{align*}
and when $\zeta_t$ is a $t^{th}$ root of unity, we have
\begin{align*}
	S_X(\zeta_t,q) 
	&=
	\sum_{n=1}^\infty \left(\sum_{k=0}^{t-1}M_X(k,t,n)\zeta_t^k\right)q^n
.
\end{align*}
The last equation is of great importance because if $t$ is prime and 
$\zeta_t$ is a primitive $t^{th}$ root of unity, then the minimal polynomial
for $\zeta_t$ is $1+x+x^2+\dots+x^{t-1}$. Thus if the coefficient of $q^N$ in
$S_X(\zeta_t,q)$ is zero, then $\sum_{k=0}^{t-1}M_X(k,t,N)\zeta_t^k$ is zero 
and so $M_X(0,t,N)=M_X(1,t,N)=\dots=M_X(t-1,t,N)$. But then we would have that
$\sptA{X}{N}=t\cdot M_X(0,t,N)$ and so clearly $\sptA{X}{N}\equiv 0\pmod{t}$.
That is to say, if the coefficient of $q^N$ in $S_X(\zeta_t,q)$ is zero, then
$\sptA{X}{N}\equiv 0 \pmod{t}$. Thus not only do we have the congruence
$\sptA{X}{N}\equiv 0 \pmod{t}$, but also the stronger combinatorial result that
all of the $M_{X}(r,t,N)$ are equal.

We return to this idea after defining our
new spt-type functions and listing their congruences in the next section.
There we introduce new spt-crank-type functions corresponding to the Bailey pairs
$C(1)$, $C(5)$, $E(2)$, and $E(4)$ of \cite{Slater} as well as revisit the Bailey pairs
$A(1)$, $A(3)$, $A(5)$, $A(7)$. Not only do we prove congruences for these
functions by dissection formulas for $S_{X}(\zeta,q)$, but we find single series and product identities
for $S_{A5}(z,q)$, $S_{A7}(z,q)$, $S_{C5}(z,q)$, and $S_{E2}(z,q)$, and find interesting 
Hecke-Rogers-type double sum formulas
for $S_{A1}(z,q)$, $S_{A3}(z,q)$, $S_{C1}(z,q)$, and $S_{E4}(z,q)$. These
series identities can be used to prove
most, but not all, of the congruences whereas the dissections prove all of them. 

In \cite{Garvan3} the first author gave Hecke-Rogers-type double series for the original
spt-crank functions.
In particular we have,
\begin{align*}
	&\aqprod{z,z^{-1},q}{q}{\infty}S(z,q)
	=
	\sum_{n=0}^\infty\sum_{m=0}^n
		(1-z^{\frac{n-m}{2}})^2z^{\frac{m-n}{2}}
		\Jac{-4}{n}\Jac{12}{m} q^{\frac{1}{12}\left(\frac{3n^2-m^2}{2}-1\right)}
	,\\
	&(1+z)\aqprod{z,z^{-1},q}{q}{\infty}\SB(z,q)
	\\
	&=	
	\sum_{n=0}^\infty\sum_{m=-[n/2]}^{[n/2]}
		(-1)^{m+n}(1-z^{n-2|m|+1})(1-z^{n-2|m|})z^{2|m|-n}q^{\frac{n^2-2m^2}{2} + \frac{n}{2}}
	,\\
	&\aqprod{z,z^{-1},q^2}{q^2}{\infty}\STwoB(z,-q)
	=
		\sum_{n=0}^\infty\sum_{m=0}^n
		(-1)^n(1-z^{n-m})^2z^{m-n}q^{\frac{2n^2-m^2}{2} + \frac{2n-m}{2}}
.
\end{align*}
Here $\Jac{\cdot}{\cdot}$ is the Kronecker symbol.
Such identities are interesting not only because of their use in proving 
congruences for smallest parts functions, but also because many mock theta
functions are special cases of the ranks related to the spt-crank functions.

In Section 2 we give the preliminaries and state our main results which are 
congruences for the spt-type functions
and identities for the spt-crank-type functions including 
single series and Hecke-Rogers-type 
double series identities, product identities, 
and dissection identities. In 
Section 3 we use the machinery of Bailey pairs to prove the series and product 
identities. In Section 4 we evaluate the spt-crank-type functions at roots 
of unity to obtain identities for the $M_X(r,t,n)$ and prove the congruences. In 
Section 5 we relate our spt-crank-type functions to known rank and crank 
functions and derive the dissection identities. 
In Section 6 we finish with a few concluding remarks.
We summarize the results of this article in the following table:
\begin{center}
\begin{tabular}{|c|c|c|c|c|}
	\hline
	Bailey pair 	& linear congruence 	& Hecke-Rogers 		& product  		& dissection  
	\\
	$X$ 			&  mod $p$ for			& identity for  	& identity for 	& identity for 
	\\
					&	$S_X(1,q)$			& $S_X(z,q)$		& $S_X(z,q)$	& $S_X(\zeta_p,q)$
	\\
	\hline
	$A1$ 			& $p=3$ 				& Yes 				& No 			& In \cite{JS}
	\\
	$A3$ 			& $p=3,5$ 				& Yes 				& No 			& In \cite{JS}
	\\
	$A5$ 			& $p=5,7$				& No			 	& Yes 			& In \cite{JS}
	\\
	$A7$ 			& $p=5$					& No 				& Yes 			& In \cite{JS}
	\\		
	$C1$ 			& $p=5$					& Yes 				& No 			& Yes
	\\
	$C5$ 			& $p=5$					& No 				& Yes 			& Yes
	\\
	$E2$ 			& $p=3$					& No 				& Yes 			& Yes
	\\
	$E4$ 			& $p=3$					& Yes 				& No 			& Yes
	\\
	\hline
\end{tabular}
\end{center}

\section{Preliminaries and Statement of Results}

To begin we take the following Bailey pairs relative to $(1,q)$
from \cite{Slater}:
\begin{align*}
	\beta^{A1}_n &= \frac{1}{\aqprod{q}{q}{2n}}
	,&
	\alpha^{A1}_n &=
	\left\{\begin{array}{ll}
		1 & \mbox{ if } n=0
		\\
		 q^{6k^2-k} + q^{6k^2+k}& \mbox{ if } n = 3k
		\\
		-q^{6k^2+ 5k + 1}& \mbox{ if } n = 3k+1
		\\
		-q^{6k^2 - 5k + 1}& \mbox{ if } n = 3k-1
	\end{array}\right.,
	\\
	\beta^{A3}_n &= \frac{q^n}{\aqprod{q}{q}{2n}}
	,&
	\alpha^{A3}_n &=
	\left\{\begin{array}{ll}
		1 & \mbox{ if } n=0
		\\
		 q^{6k^2-2k} + q^{6k^2+2k}& \mbox{ if } n = 3k
		\\
		-q^{6k^2+ 2k}& \mbox{ if } n = 3k+1
		\\
		-q^{6k^2 - 2k}& \mbox{ if } n = 3k-1
	\end{array}\right.,
	\\
	\beta^{A5}_n &= \frac{q^{n^2}}{\aqprod{q}{q}{2n}}
	,&
	\alpha^{A5}_n &=
	\left\{\begin{array}{ll}
		1 & \mbox{ if } n=0
		\\
		 q^{3k^2-k} + q^{3k^2+k}& \mbox{ if } n = 3k
		\\
		-q^{3k^2+ k}& \mbox{ if } n = 3k+1
		\\
		-q^{3k^2 - k}& \mbox{ if } n = 3k-1
	\end{array}\right.,
	\\
	\beta^{A7}_n &= \frac{q^{n^2-n}}{\aqprod{q}{q}{2n}}
	,&
	\alpha^{A7}_n &=
	\left\{\begin{array}{ll}
		1 & \mbox{ if } n=0
		\\
		 q^{3k^2-2k} + q^{3k^2+2k}& \mbox{ if } n = 3k
		\\
		-q^{3k^2+ 4k+1}& \mbox{ if } n = 3k+1
		\\
		-q^{3k^2 - 4k+1}& \mbox{ if } n = 3k-1
	\end{array}\right.,
	\\
	\beta^{C1}_n &= \frac{1}{\aqprod{q}{q^2}{n}\aqprod{q}{q}{n}}
	,&
	\alpha^{C1}_n &= 
	\left\{\begin{array}{ll}
		1 & \mbox{ if } n=0
		\\
		(-1)^kq^{3k^2-k}(1+q^{2k}) & \mbox{ if } n = 2k
		\\
		0 & \mbox{ if } n = 2k+1
	\end{array}\right.,
	\\
	\beta^{C5}_n &= \frac{q^{(n^2-n)/2}}{\aqprod{q}{q^2}{n}\aqprod{q}{q}{n}}
	,&
	\alpha^{C5}_n &= 
	\left\{\begin{array}{ll}
		1 & \mbox{ if } n=0
		\\
		(-1)^kq^{k^2-k}(1+q^{2k}) & \mbox{ if } n = 2k
		\\
		0 & \mbox{ if } n = 2k+1
	\end{array}\right.,
	\\
	\beta^{E2}_n &= \frac{(-1)^n}{\aqprod{q^2}{q^2}{n}}
	,&
	\alpha^{E2}_n &= 
	\left\{\begin{array}{ll}
		1 & \mbox{ if } n=0
		\\
		2(-1)^n & \mbox{ if } n \ge 1
	\end{array}\right.,
	\\
	\beta^{E4}_n &= \frac{q^n}{\aqprod{q^2}{q^2}{n}}
	,&
	\alpha^{E4}_n &= 
	\left\{\begin{array}{ll}
		1 & \mbox{ if } n=0
		\\
		(-1)^n q^{n^2-n}(1+q^{2n}) & \mbox{ if } n \ge 1
	\end{array}\right.
.
\end{align*}
For each Bailey pair we define a two variable
spt-crank-type series as follows,
\begin{align}
	S_{A1}(z,q) 
	&=
		\frac{\aqprod{q}{q}{\infty}}{\aqprod{z,z^{-1}}{q}{\infty}}		
		\sum_{n=1}^\infty \frac{q^n\aqprod{z,z^{-1}}{q}{n}}{\aqprod{q}{q}{2n}}
		=
		\sum_{n=1}^\infty \frac{q^{n}\aqprod{q^{2n+1}}{q}{\infty}}
			{\aqprod{zq^n,z^{-1}q^n}{q}{\infty}}
	,\\
	S_{A3}(z,q) 
	&=
		\frac{\aqprod{q}{q}{\infty}}{\aqprod{z,z^{-1}}{q}{\infty}}		
		\sum_{n=1}^\infty \frac{q^{2n}\aqprod{z,z^{-1}}{q}{n}}{\aqprod{q}{q}{2n}}
		=
		\sum_{n=1}^\infty \frac{q^{2n}\aqprod{q^{2n+1}}{q}{\infty}}
			{\aqprod{zq^n,z^{-1}q^n}{q}{\infty}}
	,\\
	S_{A5}(z,q) 
	&=
		\frac{\aqprod{q}{q}{\infty}}{\aqprod{z,z^{-1}}{q}{\infty}}		
		\sum_{n=1}^\infty \frac{q^{n^2+n}\aqprod{z,z^{-1}}{q}{n}}{\aqprod{q}{q}{2n}}
		=
		\sum_{n=1}^\infty \frac{q^{n^2+n}\aqprod{q^{2n+1}}{q}{\infty}}
			{\aqprod{zq^n,z^{-1}q^n}{q}{\infty}}
	,\\
	S_{A7}(z,q) 
	&=
		\frac{\aqprod{q}{q}{\infty}}{\aqprod{z,z^{-1}}{q}{\infty}}		
		\sum_{n=1}^\infty \frac{q^{n^2}\aqprod{z,z^{-1}}{q}{n}}{\aqprod{q}{q}{2n}}
		=
		\sum_{n=1}^\infty \frac{q^{n^2}\aqprod{q^{2n+1}}{q}{\infty}}
			{\aqprod{zq^n,z^{-1}q^n}{q}{\infty}}
	,\\
	S_{C1}(z,q) 
	&=
		\frac{\aqprod{q}{q^2}{\infty}\aqprod{q}{q}{\infty}}{\aqprod{z,z^{-1}}{q}{\infty}}		
		\sum_{n=1}^\infty \frac{q^n\aqprod{z,z^{-1}}{q}{n}}{\aqprod{q}{q^2}{n}\aqprod{q}{q}{n}}
		=
		\sum_{n=1}^\infty \frac{q^{n}\aqprod{q^{2n+1}}{q^2}{\infty}\aqprod{q^{n+1}}{q}{\infty}}
			{\aqprod{zq^n,z^{-1}q^n}{q}{\infty}}
	,\\
	S_{C5}(z,q) 
	&=
		\frac{\aqprod{q}{q^2}{\infty}\aqprod{q}{q}{\infty}}{\aqprod{z,z^{-1}}{q}{\infty}}		
		\sum_{n=1}^\infty \frac{q^{\frac{n^2+n}{2}}\aqprod{z,z^{-1}}{q}{n}}{\aqprod{q}{q^2}{n}\aqprod{q}{q}{n}}
		=
		\sum_{n=1}^\infty \frac{q^{\frac{n^2+n}{2}}\aqprod{q^{2n+1}}{q^2}{\infty}\aqprod{q^{n+1}}{q}{\infty}}
			{\aqprod{zq^n,z^{-1}q^n}{q}{\infty}}
	,\\
	S_{E2}(z,q) 
	&=
		\frac{\aqprod{q^2}{q^2}{\infty}}{\aqprod{z,z^{-1}}{q}{\infty}}		
		\sum_{n=1}^\infty \frac{(-1)^nq^n\aqprod{z,z^{-1}}{q}{n}}{\aqprod{q^2}{q^2}{n}}
		=	
		\sum_{n=1}^\infty \frac{(-1)^nq^{n}\aqprod{q^{2n+2}}{q^2}{\infty}}
			{\aqprod{zq^n,z^{-1}q^n}{q}{\infty}}
	,\\
	S_{E4}(z,q) 
	&=
		\frac{\aqprod{q^2}{q^2}{\infty}}{\aqprod{z,z^{-1}}{q}{\infty}}		
		\sum_{n=1}^\infty \frac{q^{2n}\aqprod{z,z^{-1}}{q}{n}}{\aqprod{q^2}{q^2}{n}}
		=	
		\sum_{n=1}^\infty \frac{q^{2n}\aqprod{q^{2n+2}}{q^2}{\infty}}
			{\aqprod{zq^n,z^{-1}q^n}{q}{\infty}}
.
\end{align}
Setting $z=1$ gives the following spt-type functions
\begin{align*}
	S_{A1} (q)
	&=
	\sum_{n=1}^\infty \sptA{A1}{n}q^n
	=
	\sum_{n=1}^\infty \frac{q^n}
	{(1-q^n)^2\aqprod{q^{n+1}}{q}{\infty}\aqprod{q^{n+1}}{q}{n}}
	,\\			
	S_{A3} (q)
	&=
	\sum_{n=1}^\infty \sptA{A3}{n}q^n
	=
	\sum_{n=1}^\infty \frac{q^{2n}}
	{(1-q^n)^2\aqprod{q^{n+1}}{q}{\infty}\aqprod{q^{n+1}}{q}{n}}
	,\\
	S_{A5} (q)
	&=
	\sum_{n=1}^\infty \sptA{A5}{n}q^n
	=
	\sum_{n=1}^\infty \frac{q^{n^2+n}}
	{(1-q^n)^2\aqprod{q^{n+1}}{q}{\infty}\aqprod{q^{n+1}}{q}{n}}
	,\\
	S_{A7} (q)
	&=
	\sum_{n=1}^\infty \sptA{A7}{n}q^n
	=
	\sum_{n=1}^\infty \frac{q^{n^2}}
	{(1-q^n)^2\aqprod{q^{n+1}}{q}{\infty}\aqprod{q^{n+1}}{q}{n}}
	,\\
	S_{C1} (q)
	&=
	\sum_{n=1}^\infty \sptA{C1}{n}q^n
	=
	\sum_{n=1}^\infty \frac{q^{n}}
	{(1-q^n)^2\aqprod{q^{n+1}}{q}{n}\aqprod{q^{2n+2}}{q^2}{\infty}}
	,\\
	S_{C5} (q)
	&=
	\sum_{n=1}^\infty \sptA{C5}{n}q^n
	=
	\sum_{n=1}^\infty \frac{q^{(n^2+n)/2}}
	{(1-q^n)^2\aqprod{q^{n+1}}{q}{n}\aqprod{q^{2n+2}}{q^2}{\infty}}
	,\\
	S_{E2} (q)
	&=
	\sum_{n=1}^\infty \sptA{E2}{n}q^n
	=
	\sum_{n=1}^\infty \frac{(-1)^nq^{n}\aqprod{-q^{n+1}}{q}{\infty}}
	{(1-q^n)^2\aqprod{q^{n+1}}{q}{\infty}}
	,\\
	S_{E4} (q)
	&=
	\sum_{n=1}^\infty \sptA{E4}{n}q^n
	=
	\sum_{n=1}^\infty \frac{q^{2n}\aqprod{-q^{n+1}}{q}{\infty}}
	{(1-q^n)^2\aqprod{q^{n+1}}{q}{\infty}}
.
\end{align*}
For group A, the interpretation is a bit more natural in terms of partition 
pairs rather than a smallest parts function, however we give the smallest parts
interpretation for each of these generating functions. 

For a partition $\pi$ we let $s(\pi)$ denote the smallest part, $\ell(\pi)$ the 
largest part, and $|\pi|$ the sum of parts. We say a pair of partitions 
$(\pi_1,\pi_2)$ is a partition pair of $n$ if $|\pi_1|+|\pi_2|=n$.
We let $PP$ denote the set of all partition pairs $(\pi_1,\pi_2)$
such that $\pi_1$ is non-empty and all parts of $\pi_2$ are larger than
$s(\pi_1)$ but no more than $2s(\pi_1)$. We let
$\widetilde{P}$ denote the set of all partitions $\pi$ such that
all odd parts of $\pi$ are less than $2s(\pi)$.
We let $\overline{P}$ denote the set of all overpartitions $\pi$
where the smallest part of $\pi$ is not overlined.
We recall that
\begin{align*}
	\frac{q^n}{(1-q^n)^2} &= q^n + 2q^{2n} + 3q^{3n} + 4q^{4n} + \dots
\end{align*}
and so in the generating functions we can interpret this as contributing
the smallest part and weighted by the number of times the smallest part
appears.

\begin{definition}
Suppose $n\ge 0$ is an integer. 
We let $\sptA{A1}{n}$ denote the weighted sum of the partition pairs 
$(\pi_1,\pi_2)\in PP$ of $n$, where the weight is given by the number of times 
$s(\pi_1)$ appears.
We let $\sptA{A3}{n}$ denote the weighted sum of the partition pairs
$(\pi_1,\pi_2)\in PP$ of $n$, where the weight is given by one less than the 
number of times $s(\pi_1)$ appears.
We let $\sptA{A5}{n}$ denote the weighted sum of the partition pairs 
$(\pi_1,\pi_2)\in PP$ of $n$, where the weight is given by the number of times 
$s(\pi_1)$ appears past the first $s(\pi_1)$ times. 
Lastly, we let $\sptA{A7}{n}$ denote the weighted sum of the partition pairs 
$(\pi_1,\pi_2)\in PP$ of $n$, where the weight is given by the number of times
$s(\pi_1)$ appears past the first $s(\pi_1)-1$ times.
\end{definition}

\begin{definition}
Suppose $n\ge 0$ is an integer. We let
$\sptA{C1}{n}$ denote the weighted sum of partitions 
$\pi\in\widetilde{P}$ of $n$, 
where the weight is given by the number of times $s(\pi)$ appears.
We let
$\sptA{C5}{n} = \sptA{C1}{n}-\sptA{}{n/2}$, where $\sptA{}{n/2}$ is zero if 
$n$ is odd.
\end{definition}

\begin{definition}
Suppose $n\ge 0$ is an integer. We let
$\sptA{E2}{n}$ denote the weighted sum of overpartitions 
$\pi\in\overline{P}$ of $n$,
where the weight is given by
$(-1)^{s(\pi)}$ times the number of times that $s(\pi)$ appears.
We let $\sptA{E4}{n}$ denote the weighted sum overpartitions 
$\pi\in\overline{P}$ of $n$, 
where the weight is given by
by one less than the number of times $s(\pi)$ appears.
\end{definition}

The definition of $\sptA{C5}{n}$ will be justified by
Corollary \ref{CorollaryC1AndC5Identities}.
We note 
$\sptA{E2}{n} = \sptBar{2}{n} - \sptBar{1}{n}$, where $\sptBar{2}{n}$ and 
$\sptBar{1}{n}$ are two spt functions studied in \cite{GarvanJennings};
in fact $S_{E2}(z,q)=\SB_2(z,q)-\SB_1(z,q)$ where $\SB_2(z,q)$ and 
$\SB_1(z,q)$ are the two variable generalizations of the generating functions
for $\sptBar{2}{n}$ and $\sptBar{1}{n}$ that we studied in 
\cite{GarvanJennings}.
Additionally we  have
$\sptA{E4}{n} = \sptBar{}{n}-\frac{\overline{p}(n)}{2}$,
where $\overline{p}(n)$ is the number of overpartitions of $n$.
These functions satisfy the following congruences.

\begin{theorem}\label{TheoremCongruences}
\begin{align*}
	\sptA{A1}{3n}\equiv 0\pmod{3}
	,\\
	\sptA{A3}{3n+1}\equiv 0\pmod{3}
	,\\
	\sptA{A3}{5n+1}\equiv 0\pmod{5}
	,\\
	\sptA{A5}{5n+4}\equiv 0\pmod{5}
	,\\
	\sptA{A5}{7n+1}\equiv 0\pmod{7}
	,\\
	\sptA{A7}{5n+4}\equiv 0\pmod{5}
	,\\
	\sptA{C1}{5n+3}\equiv 0\pmod{5}
	,\\
	\sptA{C5}{5n+3}\equiv 0\pmod{5}
	,\\
	\sptA{E2}{3n}\equiv 0\pmod{3}
	,\\
	\sptA{E4}{3n+1}\equiv 0\pmod{3}
.
\end{align*}
\end{theorem}

We use the spt-crank-type functions to prove these congruences as 
explained in the introduction. 
We prove Theorem \ref{TheoremCongruences} by showing 
the following terms are zero:
$q^{3n}$ in $S_{A1}(\zeta_3,q)$,
$q^{3n+1}$ in $S_{A3}(\zeta_3,q)$,
$q^{5n+1}$ in $S_{A3}(\zeta_5,q)$,
$q^{5n+4}$ in $S_{A5}(\zeta_5,q)$,
$q^{7n+1}$ in $S_{A5}(\zeta_7,q)$,
$q^{5n+4}$ in $S_{A7}(\zeta_5,q)$,
$q^{5n+4}$ in $S_{C1}(\zeta_5,q)$,
$q^{5n+3}$ in $S_{C5}(\zeta_5,q)$,
$q^{3n}$ in $S_{E2}(\zeta_3,q)$,
and $q^{3n+1}$ in $S_{E4}(\zeta_3,q)$.
For convenience, if $F(x)$ is a series in $x$, then we let 
$[x^N]F(x)$ denote the coefficient of $x^N$ in $F(x)$.

That these coefficients are zero is a stronger result than just the congruences
alone. These coefficients being zero gives a manner in which to split up the
numbers $\sptA{X}{an+b}$ to see the congruences. In each case, this
would allow us to define a so called spt-crank based off of the 
$M_{X}(m,n)$. An initial interpretation of the $M_{X}(m,n)$ would be in terms
of weighted vector partitions, to find an interpretation in terms of smallest parts
would take considerable work. One can find an example of this process
in Section 3 of \cite{GarvanJennings}, we do not pursue this idea here.
We do note, however, that since
$S_{E2}(z,q) = \SB_2(z,q)-\SB_1(z,q)$, we do have a combinatorial
interpretation of $M_{E2}(m,n)$ as being the difference of the corresponding
spt-cranks defined in \cite{GarvanJennings}. Also an interpretation in 
terms of a crank defined on partition pairs for $M_{A1}(m,n)$, $M_{A3}(m,n)$,
$M_{A5}(m,n)$, and $M_{A7}(m,n)$ was given by the second author in
\cite{JS}.

We prove the following series representations for the $S_X(z,q)$.
\begin{theorem}\label{TheoremFinalSeriesIdentities}
\begin{align}
	&(1+z)\aqprod{z,z^{-1}}{q}{\infty}
	S_{A1}(z,q)
		\nonumber\\\label{FinalA1Series}&
		=
		\frac{1}{\aqprod{q}{q}{\infty}}
		\sum_{k=1}^\infty (1-z^{k-1})(1-z^k)z^{1-k}
		\Big( (-1)^{k+1}q^{k(k-1)/2} 
			+
			\sum_{n=1}^\infty (-1)^{n+k+1} q^{\frac{k(k-1)}{2} + \frac{n(n-3)}{2} + 2kn}(1+q^n)
		\Big)
	,\\
	&(1+z)\aqprod{z,z^{-1}}{q}{\infty}
	S_{A3}(z,q)
		\nonumber\\\label{FinalA3Series}&
		=
		\frac{1}{\aqprod{q}{q}{\infty}}
		\sum_{k=1}^\infty \sum_{n=0}^\infty
			(1-z^{k-1})(1-z^k)z^{1-k}(-1)^{n+k+1} q^{\frac{k(k+1)}{2} + \frac{n(n-3)}{2} + 2kn - 1}(1-q^{2n+1})
	,\\
	\label{A5Series}
	&(1+z)\aqprod{z,z^{-1}}{q}{\infty}
	S_{A5}(z,q)
	=
		\sum_{k=-\infty}^\infty (-1)^{k} (1-z^k)(1-z^{k+1})z^{-k}q^{\frac{k(3k+1)}{2}}	
	,\\
	\label{A7Series}
	&(1+z)\aqprod{z,z^{-1}}{q}{\infty}
	S_{A7}(z,q)
	=
		\sum_{k=-\infty}^\infty (-1)^{k} (1-z^k)(1-z^{k+1})z^{-k}q^{\frac{k(3k-1)}{2}}	
	,\\
	&(1+z)\aqprod{z,z^{-1}}{q}{\infty}
	S_{C1}(z,q)
		\nonumber\\\label{FinalC1Series}&
		=
		\frac{1}{\aqprod{q}{q}{\infty}}
		\sum_{k=1}^\infty\sum_{n=0}^\infty
		(1-z^{k-1})(1-z^k)z^{1-k}(-1)^{k+1} q^{\frac{k(k-1)}{2} + \frac{n(3n-1)}{2} + 3kn}
			(1-q^{2k-1})(1-q^{k+n})
	,\\
	\label{C5Series}
	&(1+z)\aqprod{z,z^{-1}}{q}{\infty}
	S_{C5}(z,q)
	=
		\sum_{k=-\infty}^\infty (1-z^{k-1})(1-z^{k})z^{1-k} 
		(-1)^{k}q^{k^2}
	,\\
	&
	\label{E2Series}
	(1+z)\aqprod{z,z^{-1}}{q}{\infty}
	S_{E2}(z,q)
	=
		\aqprod{q}{q^2}{\infty}
		\sum_{k=1}^\infty (1-z^k)(1-z^{k-1})z^{1-k}q^{\frac{k(k-1)}{2}}
	,\\
	&(1+z)\aqprod{z,z^{-1}}{q}{\infty}
	S_{E4}(z,q)
		\nonumber\\\label{FinalE4Series}
		&=
		\frac{1}{\aqprod{q}{q}{\infty}}
		\sum_{k=1}^\infty\sum_{n=0}^\infty
		(1-z^{k-1})(1-z^k)z^{1-k}(-1)^{k+n+1} q^{\frac{k(k+1)}{2} + n^2-n +2kn -1}(1-q^{2n+1})
	.
\end{align}
\end{theorem}

By letting $z$ be the appropriate root of unity, we use Theorem 
\ref{TheoremFinalSeriesIdentities} to prove all of the congruences in Theorem 
\ref{TheoremCongruences} except for $\sptA{A3}{5n+3}\equiv 0\pmod{5}$
and $\sptA{C1}{5n+3}\equiv 0\pmod{5}$. Using a later identity will prove that
$\sptA{C1}{5n+3}\equiv 0\pmod{5}$, however we will not reprove the congruence
$\sptA{A3}{5n+3}\equiv 0\pmod{5}$.
From Theorem \ref{TheoremFinalSeriesIdentities} we will also deduce the following
Hecke-Rogers-type double series for $S_{A1}(z,q)$, $S_{A3}(z,q)$, $S_{C1}(z,q)$, and 
$S_{E4}(z,q)$.
\begin{corollary}\label{CorollaryHeckeDoubleSeries}
\begin{align}
	\label{HeckeDoubleA1Series}
	&(1+z)\aqprod{z,z^{-1},q}{q}{\infty}
	S_{A1}(z,q)
	\nonumber\\
	&=
		\sum_{k=0}^\infty
		\sum_{n=-[k/2]}^{[k/2]}
		(-1)^{n+k}(1-z^{k-2|n|})(1-z^{2|n|-k+1})
		q^{\frac{k^2-k-3n^2-n}{2}}
	,\\
	\label{HeckeDoubleA3Series}
	&(1+z)\aqprod{z,z^{-1},q}{q}{\infty}
	S_{A3}(z,q)
	\nonumber\\
	&=
		\sum_{k=1}^\infty
		\sum_{n=1}^{[k/2]}
		(-1)^{n+k}(1-z^{k-2n+1})(1-z^{2n-k})q^{\frac{k^2-k-3n^2+n}{2}}
		\nonumber\\&\quad
		-
		\sum_{k=1}^\infty
		\sum_{n=0}^{[k/2]}
		(-1)^{n+k}(1-z^{k-2n})(1-z^{2n-k+1})q^{\frac{k^2+k-3n^2-n}{2}}
	\\
	\label{HeckeDoubleC1Series}
	&(1+z)\aqprod{z,z^{-1},q}{q}{\infty}
	S_{C1}(z,q)
	\nonumber\\
	&=
		\sum_{k=1}^\infty
		\sum_{n=0}^{[k/3]}
		(-1)^{n+k}(1-z^{3n-k+1})(1-z^{k-3n})q^{\frac{k^2-k}{2}-3n^2+n}
		\nonumber\\&\quad
		-
		\sum_{k=1}^\infty
		\sum_{n=0}^{[k/3]}
		(-1)^{n+k}(1-z^{3n-k+1})(1-z^{k-3n})q^{\frac{k^2+k}{2}-3n^2-n}
		\nonumber\\&\quad
		+
		\sum_{k=1}^\infty
		\sum_{n=1}^{[k/3]}
		(-1)^{n+k}(1-z^{3n-k})(1-z^{k-3n+1})q^{\frac{k^2-k}{2}-3n^2+n}
		\nonumber\\&\quad
		-
		\sum_{k=1}^\infty
		\sum_{n=1}^{[k/3]}
		(-1)^{n+k}(1-z^{3n-k})(1-z^{k-3n+1})q^{\frac{k^2+k}{2}-3n^2-n}
	,\\
	\label{HeckeDoubleE4Series}
	&(1+z)\aqprod{z,z^{-1},q}{q}{\infty}
	S_{E4}(z,q)
	\nonumber\\
	&=
		\sum_{k=1}^\infty
		\sum_{n=1}^{[k/2]}
		(-1)^{n+k}(1-z^{2n-k})(1-z^{k-2n+1})q^{\frac{k^2-k}{2}-n^2}
		\nonumber\\&\quad
		-
		\sum_{k=1}^\infty
		\sum_{n=0}^{[k/2]}
		(-1)^{n+k}(1-z^{2n-k+1})(1-z^{k-2n})q^{\frac{k^2+k}{2}-n^2}
.
\end{align}
\end{corollary}

Using the Jacobi Triple Product Identity \cite{Andrews2}
in the forms
\begin{align*}
	\aqprod{zq,z^{-1}q,q^2}{q^2}{\infty}
	&=
	\sum_{n=-\infty}^\infty (-1)^n z^n q^{n^2}
	,\\
	\aqprod{z,z^{-1}q,q}{q}{\infty}
	&=
	\sum_{n=-\infty}^\infty (-1)^n z^n q^{n(n-1)/2}
	,	
\end{align*}
we see Theorem \ref{TheoremFinalSeriesIdentities} gives the remarkable fact 
that $S_{A1}(z,q)$, $S_{A3}(z,q)$, $S_{C5}(z,q)$, and $S_{E2}(z,q)$
can be written entirely in terms of products. Based on the
initial series definitions of
$S_{A1}(z,q)$, $S_{A3}(z,q)$, $S_{C5}(z,q)$, and $S_{E2}(z,q)$
these 
identities are actually known identities in disguise
(see Chapter 3 and Appendix II of 
\cite{GasperRahman}).

\begin{corollary}\label{CorollaryTwoVariableProductIdentities}
\begin{align*}
	S_{A5}(z,q)
	&=
		\frac{z\aqprod{zq^2,z^{-1}q,q^3}{q^3}{\infty}}
			{(1+z)\aqprod{z,z^{-1}}{q}{\infty}}
		+
		\frac{\aqprod{zq,z^{-1}q^2,q^3}{q^3}{\infty}}
			{(1+z)\aqprod{z,z^{-1}}{q}{\infty}}
		-
		\frac{\aqprod{q}{q}{\infty}}
			{\aqprod{z,z^{-1}}{q}{\infty}}
	,\\
	S_{A7}(z,q)
	&=
		\frac{z\aqprod{zq,z^{-1}q^2,q^3}{q^3}{\infty}}
			{(1+z)\aqprod{z,z^{-1}}{q}{\infty}}
		+
		\frac{\aqprod{zq^2,z^{-1}q,q^3}{q^3}{\infty}}
			{(1+z)\aqprod{z,z^{-1}}{q}{\infty}}
		-
		\frac{\aqprod{q}{q}{\infty}}
			{\aqprod{z,z^{-1}}{q}{\infty}}
	,\\
	S_{C5}(z,q)
	&=
		\frac{\aqprod{zq,z^{-1}q,q^2}{q^2}{\infty}}
			{\aqprod{z,z^{-1}}{q}{\infty}}
		-
		\frac{\aqprod{q}{q}{\infty}}
			{\aqprod{-q,z,z^{-1}}{q}{\infty}}
	,\\
	S_{E2}(z,q)
	&=
		\frac{\aqprod{-zq,-z^{-1}q,q}{q}{\infty}}
			{\aqprod{z,z^{-1},-q}{q}{\infty}}
		-
		\frac{\aqprod{q^2}{q^2}{\infty}}
			{\aqprod{z,z^{-1}}{q}{\infty}}
.
\end{align*}
\end{corollary}

We prove the following dissections of
$S_{C1}(\zeta_5,q)$, $S_{C5}(\zeta_5,q)$, $S_{E2}(\zeta_3,q)$, and
$S_{E4}(\zeta_3,q)$. The dissections for $S_{A1}(\zeta_3,q)$,
$S_{A3}(\zeta_3,q)$, $S_{A3}(\zeta_5,q)$, $S_{A5}(\zeta_5,q)$,
$S_{A5}(\zeta_7,q)$, and $S_{A7}(\zeta_5,q)$ can be found in \cite{JS}.
\begin{theorem}\label{TheoremAllDissections}
\begin{align*}
	&(1-\zeta_5)(1-\zeta_5^{-1})S_{C1}(\zeta_5,q)
	\\
	&=
		\frac{\aqprod{q^{50}}{q^{50}}{\infty}\jacprod{q^{20}}{q^{50}}}
			{\jacprod{q^{10}}{q^{50}}^2}
		-
		\frac{\aqprod{q^{50}}{q^{50}}{\infty}\jacprod{q^{25}}{q^{50}}}{\jacprod{q^{5}}{q^{25}}}
		\\&\quad
		-
		2(\zeta_5+\zeta_5^{-1})q^5\frac{\aqprod{q^{50}}{q^{50}}{\infty}\jacprod{q^{5}}{q^{50}}}
			{\jacprod{q^{10}}{q^{25}}}	
		+
		q^{10}\frac{\zeta_5+\zeta_5^{-1}-2}{\aqprod{q^{50}}{q^{50}}{\infty}}
			\sum_{n=-\infty}^\infty \frac{(-1)^nq^{75n(n+1)}}{1-q^{50n+10}}
		\\&\quad
		-
		q^6(\zeta_5+\zeta_5^{-1})\frac{\aqprod{q^{50}}{q^{50}}{\infty}\jacprod{q^{10}}{q^{50}}}
		{\jacprod{q^{20}}{q^{50}}^2}
		+
		2q\frac{\aqprod{q^{50}}{q^{50}}{\infty}\jacprod{q^{15}}{q^{50}}}{\jacprod{q^{5}}{q^{25}}}
		\\&\quad
		-
		(\zeta_5+\zeta_5^{-1})q\frac{\aqprod{q^{50}}{q^{50}}{\infty}\jacprod{q^{25}}{q^{50}}}{\jacprod{q^{10}}{q^{25}}}	
		-
		q^{16}\frac{2\zeta_5+2\zeta_5^{-1}+1}{\aqprod{q^{50}}{q^{50}}{\infty}}
			\sum_{n=-\infty}^\infty \frac{(-1)^nq^{75n(n+1)}}{1-q^{50n+20}}
		\\&\quad
		+
		q^2\frac{\aqprod{q^{50}}{q^{50}}{\infty}}{\jacprod{q^{10}}{q^{50}}}
		+
		2(\zeta_5+\zeta_5^{-1})q^2\frac{\aqprod{q^{50}}{q^{50}}{\infty}\jacprod{q^{15}}{q^{50}}}{\jacprod{q^{10}}{q^{25}}}	
		\\&\quad
		+
		q^4(\zeta_5+\zeta_5^{-1})\frac{\aqprod{q^{50}}{q^{50}}{\infty}}
		{\jacprod{q^{20}}{q^{50}}}	
		-
		2q^4\frac{\aqprod{q^{50}}{q^{50}}{\infty}\jacprod{q^{5}}{q^{50}}}{\jacprod{q^{5}}{q^{25}}}
	,\\
	&(1-\zeta_5)(1-\zeta_5^{-1})S_{C5}(\zeta_5,q)	
	\\
	&=
		\frac{\aqprod{q^{50}}{q^{50}}{\infty}\jacprod{q^{20}}{q^{50}}}
			{\jacprod{q^{10}}{q^{50}}^2}
		-\frac{\aqprod{q^{50}}{q^{50}}{\infty}\jacprod{q^{25}}{q^{50}}}
			{\jacprod{q^{5}}{q^{25}}}
		\\&\quad
		-2(\zeta_5+\zeta_5^{-1})q^5\frac{\aqprod{q^{50}}{q^{50}}{\infty}\jacprod{q^{5}}{q^{50}}}
			{\jacprod{q^{10}}{q^{25}}}	
		-
		(\zeta_5+\zeta_5^4)q^6\frac{\aqprod{q^{50}}{q^{50}}{\infty}\jacprod{q^{10}}{q^{50}}}
			{\jacprod{q^{20}}{q^{50}}^2}
		\\&\quad
		+
		2q\frac{\aqprod{q^{50}}{q^{50}}{\infty}\jacprod{q^{15}}{q^{50}}}
			{\jacprod{q^{5}}{q^{25}}}
		-
		(\zeta_5+\zeta_5^{-1})q\frac{\aqprod{q^{50}}{q^{50}}{\infty}\jacprod{q^{25}}{q^{50}}}
			{\jacprod{q^{10}}{q^{25}}}	
		\\&\quad
		+(\zeta_5+\zeta_5^4-1)q^2\frac{\aqprod{q^{50}}{q^{50}}{\infty}}
			{\jacprod{q^{10}}{q^{50}}}
		-
		2(\zeta_5+\zeta_5^{-1})q^2\frac{\aqprod{q^{50}}{q^{50}}{\infty}\jacprod{q^{15}}{q^{50}}}
			{\jacprod{q^{10}}{q^{25}}}	
		\\&\quad
		-
		(\zeta_5+\zeta_5^4+1)q^4\frac{\aqprod{q^{50}}{q^{50}}{\infty}}
			{\jacprod{q^{20}}{q^{50}}}
		-
		2q^4\frac{\aqprod{q^{50}}{q^{50}}{\infty}\jacprod{q^{5}}{q^{50}}}
			{\jacprod{q^{5}}{q^{25}}}
	,\\
	&S_{E2}(\zeta_3,q)	
	=
		-
		q\frac{\aqprod{q^{18}}{q^{18}}{\infty}\aqprod{q^9}{q^9}{\infty}}
			{\aqprod{q^3}{q^3}{\infty}}
		+
		q^2\frac{\aqprod{q^{18}}{q^{18}}{\infty}^4}
			{\aqprod{q^9}{q^9}{\infty}^2\aqprod{q^6}{q^6}{\infty}}
	,\\
	&S_{E4}(\zeta_3,q)
	=
		\frac{1}{2}
		-
		\frac{\aqprod{q^9}{q^9}{\infty}^4\aqprod{q^6}{q^6}{\infty}}
			{2\aqprod{q^{18}}{q^{18}}{\infty}^2\aqprod{q^3}{q^3}{\infty}^2}
		+
		q^2\frac{\aqprod{q^{18}}{q^{18}}{\infty}}{\aqprod{q^9}{q^9}{\infty}^2}
			\sum_{n=-\infty}^\infty\frac{(-1)^nq^{9n^2+9n}}{1-q^{9n+3}} 
.
\end{align*}
\end{theorem}
From these identities we see $S_{C1}(\zeta_5,q)$ has no $q^{5n+3}$ terms,
$S_{C5}(\zeta_5,q)$ has no $q^{5n+3}$ terms,
$S_{E2}(\zeta_3,q)$ has no $q^{3n}$ terms, and 
$S_{E4}(\zeta_3,q)$ has no $q^{3n+1}$ terms. This gives another proof of the
congruences for $\sptA{C5}{5n+3}$, $\sptA{E2}{3n}$, and
$\sptA{E4}{3n+1}$, and the only proof of the congruence for 
$\sptA{C1}{5n+3}$. The dissection from \cite{JS} gives the only proof of the
congruence for $\sptA{A3}{5n+1}$.

We note that Theorems \ref{TheoremFinalSeriesIdentities} and 
\ref{TheoremAllDissections} both prove the congruences, but the identities
are inherently different. In Theorem \ref{TheoremFinalSeriesIdentities}
we have identities that are valid for all $z$, but it may be difficult to 
identify all terms in the $\ell$-dissection when $z=\zeta_\ell$. In
Theorem \ref{TheoremAllDissections}, the identities are for $z$ being a fixed
root of unity and all terms in the dissection are explicitly determined.
The dissection formulas come as a consequence of each $S_X(z,q)$ being the 
difference of a rank-type function and a crank-type function.
A consequence of the proofs of Theorem \ref{TheoremFinalSeriesIdentities} and
Corollary \ref{CorollaryHeckeDoubleSeries} are the following product 
identities.
\begin{corollary}\label{CorollaryAllSingleSeriesProducts}
\begin{align}
	\label{CorollaryProduct1}
	\aqprod{q}{q}{\infty}^2
	&=
		\sum_{k=1}^\infty
		\left( (-1)^{k+1}q^{k(k-1)/2} +
			\sum_{n=1}^\infty (-1)^{n+k+1} q^{\frac{k(k-1)}{2} + \frac{n(n-3)}{2} + 2kn}(1+q^n)
		\right)
	,\\
	\label{CorollaryProduct2}
	\aqprod{q}{q}{\infty}^2
	&=
		\sum_{k=1}^\infty \sum_{n=0}^\infty
			(-1)^{n+k+1} q^{\frac{k(k+1)}{2} + \frac{n(n-3)}{2} + 2kn - 1}(1-q^{2n+1})
	,\\
	\label{CorollaryProduct3}
	\aqprod{q}{q}{\infty}^2\aqprod{q}{q^2}{\infty}
	&=
		\sum_{k=1}^\infty\sum_{n=0}^\infty
		(-1)^{k+1} q^{\frac{k(k-1)}{2} + \frac{n(3n-1)}{2} + 3kn}(1-q^{2k-1})(1-q^{k+n})
	,\\
	\label{CorollaryProduct4}
	\aqprod{q}{q}{\infty}\aqprod{q^2}{q^2}{\infty}
	&=
		\sum_{k=1}^\infty\sum_{n=0}^\infty
		(-1)^{k+n+1} q^{\frac{k(k+1)}{2} + n^2-n +2kn -1}(1-q^{2n+1})
	,\\
	\label{CorollaryProduct5}
	\aqprod{q}{q}{\infty}^2\aqprod{q}{q^2}{\infty}
	&=
	\sum_{k=1}^\infty\sum_{n=-[(k-1)/3]}^{[k/3]}
		(-1)^{n+k+1}q^{\frac{k^2-k}{2}-3n^2+n}(1-q^k)
	,\\
	\label{CorollaryProduct6}
	\aqprod{q}{q}{\infty}{\aqprod{q^2}{q^2}{\infty}}
	&=
		\sum_{k=0}^\infty
		\sum_{n=-[k/2]}^{[k/2]}
		(-1)^{n+k}q^{\frac{k^2+k}{2}-n^2}
.
\end{align}
\end{corollary}
We note that (\ref{CorollaryProduct6}) is known. It was derived
by Andrews in \cite{Andrews1}, Bressoud in \cite{Bressoud},
and recently reproved by the first author in \cite{Garvan3}.
The authors would like to thank Mortenson for pointing out that one can also obtain these 
corollaries by using
Equation (1.7) and Theorem 1.4
of \cite{HickersonMortenson}.
Some consequences of the proof of Theorem \ref{TheoremAllDissections} 
are the following relations between $S_{C1}(z,q)$ and $S_{C5}(z,q)$.
\begin{corollary}\label{CorollaryC1AndC5Identities}
For all $n$ and $m$,
\begin{align*}
	N_S(m,n) &= M_{C1}(m,2n) - M_{C5}(m,2n)
	,\\
	M_{C1}(m,2n+1) &= M_{C5}(m,2n+1)
	,\\
	\spt{n} &=  \sptA{C1}{2n}-\sptA{C5}{2n}
	,\\
	\sptA{C1}{2n+1} &= \sptA{C5}{2n+1}
.
\end{align*}
\end{corollary}

In Section 3 we prove the series identities of Theorem \ref{TheoremFinalSeriesIdentities}
as well as Corollaries \ref{CorollaryHeckeDoubleSeries}, 
\ref{CorollaryTwoVariableProductIdentities}, and
\ref{CorollaryAllSingleSeriesProducts}.
In Section 4 we use Theorem \ref{TheoremFinalSeriesIdentities} to prove the 
appropriate terms are zero to deduce the congruences in Theorem 
\ref{TheoremCongruences}.
In Section 5 we prove the dissection formulas of Theorem 
\ref{TheoremAllDissections}.


\section{Proof of the Series Identities}

To prove the identities for $S_{A1}(z,q)$, $S_{A3}(z,q)$,
$S_{C1}(z,q)$, and $S_{E4}(z,q)$ we need the following preliminary result. 
\begin{proposition}\label{PropositionFirstDoubleSeries}
\begin{align}
	\label{FirstA1Series}
	&(1+z)\aqprod{z,z^{-1}}{q}{\infty}
	S_{A1}(z,q)
		\nonumber\\&
		=
		-(1+z)\aqprod{q}{q}{\infty}
		\nonumber\\&
		+
		\frac{1}{\aqprod{q}{q}{\infty}}
		\sum_{k=1}^\infty (z^{k} + z^{-k+1})
		\left( (-1)^{k+1}q^{k(k-1)/2}
			+
			\sum_{n=1}^\infty (-1)^{n+k+1} q^{\frac{k(k-1)}{2} + \frac{n(n-3)}{2} + 2kn}(1+q^n)
		\right)
	,\\
	&(1+z)\aqprod{z,z^{-1}}{q}{\infty}
	S_{A3}(z,q)
		\nonumber\\\label{FirstA3Series}&
		=
		-(1+z)\aqprod{q}{q}{\infty}
		\nonumber\\&\quad
		+
		\frac{1}{\aqprod{q}{q}{\infty}}
		\sum_{k=1}^\infty \sum_{n=0}^\infty
			(z^{k} + z^{-k+1})(-1)^{n+k+1} q^{\frac{k(k+1)}{2} + \frac{n(n-3)}{2} + 2kn - 1}(1-q^{2n+1})
	,\\
	\label{FirstC1Series}
	&(1+z)\aqprod{z,z^{-1}}{q}{\infty}
	S_{C1}(z,q)
		\nonumber\\&
		=
		-(1+z)\aqprod{q}{q}{\infty}\aqprod{q}{q^2}{\infty}
		\nonumber\\&\quad
		+
		\frac{1}{\aqprod{q}{q}{\infty}}
		\sum_{k=1}^\infty\sum_{n=0}^\infty
		(z^{k}+z^{1-k})(-1)^{k+1} q^{\frac{k(k-1)}{2} + \frac{n(3n-1)}{2} + 3kn}(1-q^{2k-1})(1-q^{k+n})
	,\\
	&(1+z)\aqprod{z,z^{-1}}{q}{\infty}
	S_{E4}(z,q)
		\nonumber\\\label{FirstE4Series}
		&=
		-(1+z)\aqprod{q^2}{q^2}{\infty}
		\nonumber\\&\quad
		+
		\frac{1}{\aqprod{q}{q}{\infty}}
		\sum_{k=1}^\infty\sum_{n=0}^\infty
		(z^k+z^{1-k})(-1)^{k+n+1} q^{\frac{k(k+1)}{2} + n^2-n +2kn -1}(1-q^{2n+1})
	.
\end{align}
\end{proposition}
From this we deduce equations
(\ref{CorollaryProduct1}), (\ref{CorollaryProduct2}), (\ref{CorollaryProduct3}),
and (\ref{CorollaryProduct4})
of Corollary \ref{CorollaryAllSingleSeriesProducts}. The
product identities of Corollary \ref{CorollaryAllSingleSeriesProducts} 
along with Proposition \ref{PropositionFirstDoubleSeries}
imply equations
(\ref{FinalA1Series}), (\ref{FinalA3Series}), (\ref{FinalC1Series}), and
(\ref{FinalE4Series})
of Theorem \ref{TheoremFinalSeriesIdentities}. Equations
(\ref{A5Series}), (\ref{A7Series}), (\ref{C5Series}), and (\ref{E2Series}) do
not require this additional step, instead they only require special cases of
the Jacobi Triple Product Identity.

The proofs of the identities in Theorem \ref{TheoremFinalSeriesIdentities}
and Proposition \ref{PropositionFirstDoubleSeries}
are to 
verify the coefficients of each power of $z$ match. We do this by 
rearranging $S_X(z,q)$, extracting the coefficient of $z^k$, which is a series
in $q$, and then
using one of two general Bailey pairs with either a limiting case of Bailey's Lemma
or an identity from Bailey's Transform with a suitable conjugate Bailey pair.
We recall that $(\alpha,\beta)$ form a Bailey pair relative to $(a,q)$ if
\begin{align*}
	\beta_n &= \sum_{k=0}^n \frac{\alpha_k}{\aqprod{aq}{q}{n+k}\aqprod{q}{q}{n-k}}
\end{align*}
and Bailey's Lemma, which can be found in \cite{Bailey}, gives if 
$(\alpha,\beta)$ is a Bailey pair relative to $(a,q)$ then
\begin{align*}
	\sum_{n=0}^\infty \aqprod{\rho_1,\rho_2}{q}{n} 
		\left(\frac{aq}{\rho_1\rho_2} \right)^n \beta_n
	&=
	\frac{\aqprod{aq/\rho_1,aq/\rho_2}{q}{\infty}}{\aqprod{aq,aq/\rho_1\rho_2}{q}{\infty}}
	\sum_{n=0}^\infty \frac{
		\aqprod{\rho_1,\rho_2}{q}{n} 
		\left(\frac{aq}{\rho_1\rho_2} \right)^n \alpha_n
		}{\aqprod{aq/\rho_1,aq/\rho_2}{q}{n}}	
.  
\end{align*}

\begin{lemma}
If $\alpha$ and $\beta$ are a Bailey pair relative to $(a,q)$ then
\begin{align}
	\label{LemmaBailey1}
	&\sum_{n=0}^\infty a^n q^{n^2}\beta_n
		=
		\frac{1}{\aqprod{aq}{q}{\infty}}
		\sum_{n=0}^\infty a^n q^{n^2} \alpha_n
	,\\
	\label{LemmaBailey2}
	&\sum_{n=0}^\infty \aqprod{-\sqrt{aq}}{q}{n} a^{n/2}q^{n^2/2}\beta_n
		=
		\frac{\aqprod{-\sqrt{aq}}{q}{\infty}}{\aqprod{aq}{q}{\infty}}
		\sum_{n=0}^\infty a^{n/2}q^{n^2/2} \alpha_n
	,\\
	\label{LemmaBailey3}
	&\sum_{n=0}^\infty \aqprod{a}{q^2}{n} (-1)^n q^n\beta_n
		=
		\frac{\aqprod{aq^2}{q^2}{\infty}}{\aqprod{aq,-q}{q}{\infty}}
		\sum_{n=0}^\infty \frac{(1-a)}{(1-aq^{2n})}(-1)^n q^n \alpha_n
	,\\
	\label{LemmaBailey4}
	&\sum_{n=0}^\infty q^n\beta_n
		=
		\frac{1}{\aqprod{aq,q}{q}{\infty}}
		\sum_{n=0}^\infty\sum_{r=0}^\infty (-a)^n q^{n(n+1)/2 + 2nr +r} \alpha_r
	,\\
	\label{LemmaBailey5}
	&\sum_{n=0}^\infty q^{2n}\beta_n
		=
		\frac{1}{\aqprod{aq,q}{q}{\infty}}
		\left(
		\sum_{r=0}^\infty q^{2r}\alpha_r
		+
		\sum_{n=1}^\infty\sum_{r=0}^\infty (-1)^n a^{n-1} q^{n(n+1)/2 + 2nr}(1+aq^{2r})\alpha_r
		\right)
	,\\
	\label{LemmaBailey6}
	&\sum_{n=0}^\infty \aqprod{aq}{q^2}{n} q^{2n}\beta_n
		=
		\frac{1}{\aqprod{q}{q}{\infty}\aqprod{aq^2}{q^2}{\infty}(1+q)}
		\sum_{n=0}^\infty\sum_{r=0}^\infty (-a)^n q^{n^2 + n + 2nr +  2r}(1-q^{2n+2})\alpha_r
	.
\end{align}	
If $\alpha$ and $\beta$ are a Bailey pair relative to $(a^2q,q)$
then
\begin{align}
	\label{LemmaBailey7}
	\sum_{n=0}^\infty \aqprod{-aq}{q}{n} q^{n}\beta_n
		&=
		\frac{\aqprod{-aq}{q}{\infty}}{\aqprod{q,a^2q^2}{q}{\infty}}
		\sum_{n=0}^\infty\sum_{r=0}^\infty a^{3n} q^{n(3n+5)/2 + 3nr + r}(1-aq^{n+r+1})\alpha_r
.
\end{align}
\end{lemma}
\begin{proof}
Equation	(\ref{LemmaBailey1}) is the well known identity obtained by letting
$\rho_1,\rho_2\rightarrow\infty$ in Bailey's lemma.
For (\ref{LemmaBailey2}) we let $\rho_1\rightarrow \infty$ and 
$\rho_2 =-\sqrt{aq}$ in Bailey's Lemma, simplifying then gives the result.
For (\ref{LemmaBailey3}) we let $\rho_1 = \sqrt{a}$ and 
$\rho_2 =-\sqrt{a}$ in Bailey's Lemma, simplifying then gives the result.
Equations (\ref{LemmaBailey4}), (\ref{LemmaBailey5}), (\ref{LemmaBailey6}), and 
(\ref{LemmaBailey7}) are parts 1, 11, 12, and 13 of Theorem 1 from \cite{Lovejoy3}.
These identities are from conjugate Bailey pairs rather than specializations
of $\rho_1$ and $\rho_2$.
\end{proof}

The following are two Bailey pairs relative to $(a,q)$, both of which 
follow immediately from the definition of a Bailey pair:
\begin{align}\label{FirstBaileyPair}
	\beta_n(a) &= \frac{1}{\aqprod{aq,q}{q}{n}}
	,\\\label{FirstBaileyPairAlpha}
	\alpha_n(a) &= \PieceTwo{1}{0}{n=0}{n\ge 1}
,
\end{align}
and
\begin{align}\label{SecondBaileyPair}
	\beta^*_n(a) &= \frac{1}{\aqprod{aq^2,q}{q}{n}}
	,\\\label{SecondBaileyPairAlpha}
	\alpha^*_n(a) &= 
		\left\{
   		\begin{array}{ll}
      		1 & n=0 \\
       		-aq & n=1 \\
				0 & n\ge 2 
     		\end{array}
		\right.
.
\end{align}

In rearranging each $S_X(z,q)$, we use
Proposition 4.1 of \cite{Garvan3}, which is
\begin{align}\label{Prop4.1}
	(1+z)\aqprod{z,z^{-1}}{q}{n}
	&=
	\sum_{j=-n}^{n+1} (-1)^{j+1}\frac{\aqprod{q}{q}{2n}}
		{\aqprod{q}{q}{n+j}\aqprod{q}{q}{n-j+1}}
		(1-q^{2j-1})z^jq^{j(j-3)/2 + 1}
.
\end{align}
The proofs of the identities in Theorem \ref{TheoremFinalSeriesIdentities}
and Proposition \ref{PropositionFirstDoubleSeries} are rather uniform.
As such we will prove one identity in full detail and only give a brief
sketch of the others.

\begin{proof}[Proof of (\ref{FirstA1Series})]
By (\ref{Prop4.1}) we have
\begin{align}\label{EqA1Start}
	(1+z)\aqprod{z,z^{-1}}{q}{\infty}S_{A1}(z,q)
	&=
	\aqprod{q}{q}{\infty}\sum_{n=1}^\infty 
		\frac{(1+z)\aqprod{z,z^{-1}}{q}{n}q^{n}}{\aqprod{q}{q}{2n}}
	\nonumber\\
	&=
	\aqprod{q}{q}{\infty}
	\sum_{n=1}^\infty \sum_{j=-n}^{n+1}
		\frac{(-1)^{j+1}z^j(1-q^{2j-1})q^{n+j(j-3)/2 + 1}}{\aqprod{q}{q}{n+j}\aqprod{q}{q}{n-j+1}}
.
\end{align}

\sloppy
Since $S(z^{-1},q) = S(z,q)$, we find that the
coefficient of $z^{-j}$ in 
$(1+z)\aqprod{z,z^{-1}}{q}{\infty}S_{A1}(z,q)$ is the same as the coefficient
of $z^{j+1}$. We next find these coefficients for $j\ge 1$. We will use the
Bailey pair $\alpha$ and $\beta$ from (\ref{FirstBaileyPair}) and (\ref{FirstBaileyPairAlpha})
and apply (\ref{LemmaBailey4}).

\fussy

For $j = 1$, we take $j=1$ in (\ref{EqA1Start}) and so
$[z](1+z)\aqprod{z,z^{-1}}{q}{\infty}S_{A1}(z,q)$ is given by
\begin{align*}
	\aqprod{q}{q}{\infty}\sum_{n=1}^\infty 
	\frac{(1-q)q^n}{\aqprod{q}{q}{n+1}\aqprod{q}{q}{n}}
	&=
	\aqprod{q}{q}{\infty}\sum_{n=1}^\infty 
	\frac{q^n}{\aqprod{q^2}{q}{n}\aqprod{q}{q}{n}}
	\\
	&=
	\aqprod{q}{q}{\infty}\sum_{n=0}^\infty 
	\frac{q^n}{\aqprod{q^2}{q}{n}\aqprod{q}{q}{n}}
	-\aqprod{q}{q}{\infty}
	\\
	&=
	\aqprod{q}{q}{\infty}\sum_{n=0}^\infty q^n\beta_n(q)
	-\aqprod{q}{q}{\infty}
	\\
	&=
	\frac{\aqprod{q}{q}{\infty}}{\aqprod{q^2,q}{q}{\infty}}\sum_{n=0}^\infty 
		(-1)^n q^{n+\frac{n(n+1)}{2}}
	-\aqprod{q}{q}{\infty}
	\\
	&=
	\frac{1}{\aqprod{q^2}{q}{\infty}}\sum_{n=0}^\infty 
		(-1)^n q^{n+\frac{n(n+1)}{2}}
	-\aqprod{q}{q}{\infty}
.
\end{align*}

For $j\ge 2$, the calculations are similar. We have
$[z^j](1+z)\aqprod{z,z^{-1}}{q}{\infty}S_{A1}(z,q)$ is given by
\begin{align}
	\label{ProofA1SeriesCoeff2AndUp}
	&\aqprod{q}{q}{\infty}
	\sum_{n=j-1}^\infty
	\frac{(-1)^{j+1}(1-q^{2j-1})q^{n+j(j-3)/2 + 1}}
		{\aqprod{q}{q}{n+j}\aqprod{q}{q}{n-j+1}}
	\nonumber\\
	&=
	(-1)^{j+1} (1-q^{2j-1}) q^{j(j-1)/2} \aqprod{q}{q}{\infty}
	\sum_{n=0}^\infty \frac{q^n}{\aqprod{q}{q}{n+2j-1}\aqprod{q}{q}{n}}
	\nonumber\\
	&=
	\frac{(-1)^{j+1} (1-q^{2j-1}) q^{j(j-1)/2} 	\aqprod{q}{q}{\infty}}{\aqprod{q}{q}{2j-1}}
	\sum_{n=0}^\infty \frac{q^n}{\aqprod{q^{2j},q}{q}{n}}
	\nonumber\\
	&=
	\frac{(-1)^{j+1} (1-q^{2j-1}) q^{j(j-1)/2} 	\aqprod{q}{q}{\infty}}{\aqprod{q}{q}{2j-1}}
	\sum_{n=0}^\infty q^n\beta_n(q^{2j-1})
	\nonumber\\
	&=
	\frac{(-1)^{j+1} (1-q^{2j-1}) q^{j(j-1)/2} 	\aqprod{q}{q}{\infty}}
		{\aqprod{q}{q}{2j-1} \aqprod{q^{2j},q}{q}{\infty}}
	\sum_{n=0}^\infty (-1)^n q^{n(n-1)/2 +  2jn} 
	\nonumber\\
	&=
	\frac{(-1)^{j+1} (1-q^{2j-1}) q^{j(j-1)/2}}
		{\aqprod{q}{q}{\infty}}
	\sum_{n=0}^\infty (-1)^n q^{n(n-1)/2 +  2jn} 
.
\end{align}
We note this formula agrees at $j=1$ with the coefficient of $z$, except for the missing
$-\aqprod{q}{q}{\infty}$.

We rearrange these terms slightly. We have
\begin{align*}
	&(-1)^{j+1} (1-q^{2j-1}) q^{j(j-1)/2}
	\sum_{n=0}^\infty (-1)^n q^{n(n-1)/2 +  2jn} 
	\\
	&=
	\sum_{n=0}^\infty (-1)^{n+j+1} q^{j(j-1)/2 + n(n-1)/2 +  2jn} 
	+
	\sum_{n=0}^\infty (-1)^{n+j} q^{j(j+3)/2 + n(n-1)/2 +  2jn - 1} 
	\\
	&=
	\sum_{n=0}^\infty (-1)^{n+j+1} q^{j(j-1)/2 + n(n-1)/2 +  2jn} 
	+
	\sum_{n=1}^\infty (-1)^{n+j+1} q^{j(j-1)/2 + n(n-3)/2 +  2jn}
	\\
	&=
	(-1)^{j+1} q^{j(j-1)/2}
	+
	\sum_{n=1}^\infty (-1)^{n+j+1} q^{j(j-1)/2 + n(n-3)/2 +  2jn}(1+q^n)
.
\end{align*}

Thus we have
\begin{align*}
	&(1+z)\aqprod{z,z^{-1}}{q}{\infty}S_{A1}(z,q)
	\\
	&=
	-(1+z)\aqprod{q}{q}{\infty}
	+
	\frac{1}{\aqprod{q}{q}{\infty}}\sum_{j=1}^\infty
		(z^j+z^{1-j}) (-1)^{j+1} q^{j(j-1)/2}
	\\&\quad
	+
	\frac{1}{\aqprod{q}{q}{\infty}}\sum_{j=1}^\infty\sum_{n=1}^\infty
		(z^j+z^{1-j}) (-1)^{n+j+1} q^{j(j-1)/2 + n(n-3)/2 +  2jn}(1+q^n)
,
\end{align*}
which completes the proof of (\ref{FirstA1Series}).
\end{proof}

\begin{proof}[Proof of (\ref{FirstA3Series}),
(\ref{A5Series}), (\ref{A7Series}), (\ref{FirstC1Series}), (\ref{C5Series}),
(\ref{E2Series}), (\ref{FirstE4Series})
]

By (\ref{Prop4.1}) we have
\begin{align}\label{EqA3Start}
	(1+z)\aqprod{z,z^{-1}}{q}{\infty}S_{A3}(z,q)
	&=
	\aqprod{q}{q}{\infty}
	\sum_{n=1}^\infty \sum_{j=-n}^{n+1}
		\frac{(-1)^{j+1}z^j(1-q^{2j-1})q^{2n+j(j-3)/2 + 1}}{\aqprod{q}{q}{n+j}\aqprod{q}{q}{n-j+1}}
	,\\
	\label{EqA5Start}
	(1+z)\aqprod{z,z^{-1}}{q}{\infty}S_{A5}(z,q)
	&=
	\aqprod{q}{q}{\infty}
	\sum_{n=1}^\infty \sum_{j=-n}^{n+1}
		\frac{(-1)^{j+1}z^j(1-q^{2j-1})q^{n^2 + n +j(j-3)/2 + 1}}{\aqprod{q}{q}{n+j}\aqprod{q}{q}{n-j+1}}
	,\\
	\label{EqA7Start}
	(1+z)\aqprod{z,z^{-1}}{q}{\infty}S_{A7}(z,q)
	&=
	\aqprod{q}{q}{\infty}
	\sum_{n=1}^\infty \sum_{j=-n}^{n+1}
		\frac{(-1)^{j+1}z^j(1-q^{2j-1})q^{n^2 +j(j-3)/2 + 1}}{\aqprod{q}{q}{n+j}\aqprod{q}{q}{n-j+1}}
.
\end{align}
As with $(1+z)\aqprod{z,z^{-1}}{q}{\infty}S_{A1}(z,q)$, the coefficients
of 
$z^{-j}$
and $z^{1+j}$ are the same in each
$(1+z)\aqprod{z,z^{-1}}{q}{\infty}S_{X}(z,q)$.
For $S_{A3}(z,q)$
we use the Bailey pair $\alpha$ and $\beta$ from 
(\ref{FirstBaileyPair}) and (\ref{FirstBaileyPairAlpha})
and apply (\ref{LemmaBailey5}),
for 
$S_{A5}(z,q)$ we use the Bailey pair $\alpha$ and $\beta$ and apply (\ref{LemmaBailey1}),
and for
$S_{A7}(z,q)$ we use the Bailey pair $\alpha^*$ and $\beta^*$ and apply
(\ref{LemmaBailey1}).
The identities then follows after elementary rearrangements.

For $S_{C1}(z,q)$, $S_{C5}(z,q)$, $S_{E2}(z,q)$, and $S_{E4}(z,q)$
we first use that
\begin{align*}
	\frac{1}{\aqprod{q}{q^2}{n}\aqprod{q}{q}{n}} 
	&=
		\frac{\aqprod{q^2}{q^2}{n}}{\aqprod{q}{q}{2n}\aqprod{q}{q}{n}}
		=
		\frac{\aqprod{q,-q}{q}{n}}{\aqprod{q}{q}{2n}\aqprod{q}{q}{n}}
		=
		\frac{\aqprod{-q}{q}{n}}{\aqprod{q}{q}{2n}}
	,\\
	\frac{1}{\aqprod{q^2}{q^2}{n}} 
	&= 
	\frac{\aqprod{q}{q^2}{n}}{\aqprod{q}{q}{2n}}
,
\end{align*}
so that 
by (\ref{Prop4.1}) we have
\begin{align}
	\label{EqC1Start}
	&(1+z)\aqprod{z,z^{-1}}{q}{\infty}S_{C1}(z,q)
	\nonumber\\
		&=
		\aqprod{q}{q^2}{\infty}\aqprod{q}{q}{\infty}
		\sum_{n=1}^\infty \sum_{j=-n}^{n+1}
			\frac{(-1)^{j+1}z^j(1-q^{2j-1})q^{n + j(j-3)/2 + 1} \aqprod{-q}{q}{n}}
		{\aqprod{q}{q}{n+j}\aqprod{q}{q}{n-j+1}}
	,\\
	\label{EqC5Start}
	&(1+z)\aqprod{z,z^{-1}}{q}{\infty}S_{C5}(z,q)
		\nonumber\\
		&=
		\aqprod{q}{q^2}{\infty}\aqprod{q}{q}{\infty}
		\sum_{n=1}^\infty \sum_{j=-n}^{n+1}
			\frac{(-1)^{j+1}z^j(1-q^{2j-1})q^{n(n+1)/2 + j(j-3)/2 + 1} \aqprod{-q}{q}{n}}
		{\aqprod{q}{q}{n+j}\aqprod{q}{q}{n-j+1}}
	,\\
	\label{EqE2Start}
	&(1+z)\aqprod{z,z^{-1}}{q}{\infty}S_{E2}(z,q)
	\nonumber\\
		&=
		\aqprod{q^2}{q^2}{\infty}
		\sum_{n=1}^\infty \sum_{j=-n}^{n+1}
			\frac{(-1)^{n+j+1}z^j(1-q^{2j-1})q^{n + j(j-3)/2 + 1} \aqprod{q}{q^2}{n}}
		{\aqprod{q}{q}{n+j}\aqprod{q}{q}{n-j+1}}
	,\\
	\label{EqE4Start}
	&(1+z)\aqprod{z,z^{-1}}{q}{\infty}S_{E4}(z,q)
		\nonumber\\
		&=
		\aqprod{q^2}{q^2}{\infty}
		\sum_{n=1}^\infty \sum_{j=-n}^{n+1}
			\frac{(-1)^{j+1}z^j(1-q^{2j-1})q^{2n + j(j-3)/2 + 1} \aqprod{q}{q^2}{n}}
		{\aqprod{q}{q}{n+j}\aqprod{q}{q}{n-j+1}}
.
\end{align}
For $S_{C1}(z,q$)
we use the Bailey pair $\alpha$ and $\beta$ and apply (\ref{LemmaBailey7}),
for
$S_{C5}(z,q)$
we use the Bailey pair $\alpha$ and $\beta$ and apply (\ref{LemmaBailey2}),
for
$S_{E2}(z,q)$
we use the Bailey pair $\alpha$ and $\beta$ and apply (\ref{LemmaBailey3}),
and for
$S_{E4}(z,q)$
we use the Bailey pair $\alpha^*$ and $\beta^*$ and apply (\ref{LemmaBailey6}).
The identities then follows after elementary rearrangements.

\end{proof}

\begin{proof}[Proof of (\ref{CorollaryProduct1}), (\ref{CorollaryProduct2}),
(\ref{CorollaryProduct3}), and (\ref{CorollaryProduct4})]

As noted in (\ref{EqA1Start}),
\begin{align*}
	(1+z)\aqprod{z,z^{-1}}{q}{\infty}S_{A1}(z,q)
	&=
	\aqprod{q}{q}{\infty}\sum_{n=1}^\infty 
		\frac{(1+z)\aqprod{z,z^{-1}}{q}{n}q^{n}}{\aqprod{q}{q}{2n}}
.
\end{align*}
Thus setting $z=1$ in (\ref{FirstA1Series}) gives
\begin{align*}
	2\aqprod{q}{q}{\infty}
	&=
		\frac{1}{\aqprod{q}{q}{\infty}}
		\sum_{k=1}^\infty 2
		\left( (-1)^{k+1}q^{k(k-1)/2} +
			\sum_{n=1}^\infty (-1)^{n+k+1} q^{\frac{k(k-1)}{2} + \frac{n(n-3)}{2} + 2kn}(1+q^n)
		\right)
.
\end{align*}
This proves (\ref{CorollaryProduct1}).
Similarly (\ref{CorollaryProduct2}), (\ref{CorollaryProduct3}), and
(\ref{CorollaryProduct4}), follow by setting $z=1$ in 
(\ref{FirstA3Series}), (\ref{FirstC1Series}), and (\ref{FirstE4Series}) respectively.
\end{proof}

\begin{proof}[Proof of (\ref{FinalA1Series}), (\ref{FinalA3Series}),
(\ref{FinalC1Series}), and (\ref{FinalE4Series}) ]
We see (\ref{FinalA1Series}) follows directly from (\ref{FirstA1Series})
and (\ref{CorollaryProduct1}), noting that
$z^k + z^{1-k} - 1 - z = (1-z^{k-1})(1-z^k)z^{1-k}= (z^{1-k}-1)(1-z^k)z^{1-k}$. 
The remaining three identities of Theorem \ref{TheoremFinalSeriesIdentities} 
follow as well with the corresponding product identities product identities 
of Corollary \ref{CorollaryAllSingleSeriesProducts}.
\end{proof}

\begin{proof}[Proof for Corollary \ref{CorollaryHeckeDoubleSeries} ]
The proofs only require elementary rearrangements of series and 
Theorem \ref{TheoremFinalSeriesIdentities}. 
For (\ref{HeckeDoubleA1Series}) we have
\begin{align*}
	&\sum_{k=1}^\infty\sum_{n=-[k/2]}^{[k/2]}
	(-1)^{n+k}(1-z^{k-2|n|})(1-z^{2|n|-k+1}) q^{\frac{k^2-k-3n^2-n}{2}}
	\\
	&=
	\sum_{k=0}^\infty\sum_{n=-[k/2]}^{[k/2]}
	(-1)^{n+k}(1-z^{k-2|n|})(1-z^{2|n|-k+1}) q^{\frac{k^2-k-3n^2-n}{2}}
	\\
	&=
	\sum_{n=-\infty}^\infty\sum_{k\ge 2|n|}
	(-1)^{n+k}(1-z^{k-2|n|})(1-z^{2|n|-k+1})q^{\frac{k^2-k-3n^2-n}{2}}
	\\
	&=
	\sum_{n=-\infty}^\infty\sum_{k=0}^\infty
	(-1)^{n+k}(1-z^k)(1-z^{-k+1})q^{\frac{k^2-k+n^2-n}{2}+2k|n|-|n|}
	\\
	&=
	\sum_{n=-\infty}^\infty\sum_{k=1}^\infty
	(-1)^{n+k}(1-z^k)(1-z^{-k+1})q^{\frac{k^2-k+n^2-n}{2}+2k|n|-|n|}
	\\
	&=
	\sum_{n=-\infty}^\infty\sum_{k=1}^\infty
	(-1)^{n+k+1}(1-z^k)(1-z^{k-1})z^{1-k}q^{\frac{k^2-k+n^2-n}{2}+2k|n|-|n|}
	\\
	&=
	\sum_{k=1}^\infty
	(1-z^k)(1-z^{k-1})z^{1-k}
	\left(
		(-1)^{k+1}q^{\frac{k^2-k}{2}}
		+
		\sum_{n=1}^\infty (-1)^{n+k+1}q^{\frac{k^2-k+n^2-n}{2}+2kn-n}
		\right.\\&\qquad\left.
		+
		\sum_{n=1}^\infty (-1)^{n+k+1}q^{\frac{k^2-k+n^2+n}{2}+2kn-n}
	\right)
	\\
	&=
	\sum_{k=1}^\infty
	(1-z^k)(1-z^{k-1})z^{1-k}
	\left(
		(-1)^{k+1}q^{\frac{k^2-k}{2}}
		+
		\sum_{n=1}^\infty (-1)^{n+k+1}q^{\frac{k^2-k+n^2-3n}{2}+2kn}(1+q^n)
	\right)
	\\
	&=
	(1+z)\aqprod{q,z,z^{-1}}{q}{\infty}S_{A1}(z,q)	
.
\end{align*}

The proofs of (\ref{HeckeDoubleA3Series}),
(\ref{HeckeDoubleC1Series}),
and
(\ref{HeckeDoubleE4Series}) 
follow in the same fashions and so we omit the details.

\end{proof}

\begin{proof}[Proof of (\ref{CorollaryProduct5}) and (\ref{CorollaryProduct6})]
These identities follow from a similar set of rearrangements as those
for (\ref{HeckeDoubleC1Series}) and  
(\ref{HeckeDoubleE4Series}). As such we only include the details for
(\ref{CorollaryProduct5}).
We find that
\begin{align*}
		&\sum_{k=1}^\infty
		\sum_{n=0}^{[k/3]}
		(-1)^{n+k}(-z^{3n-k+1}-z^{k-3n})q^{\frac{k^2-k}{2}-3n^2+n}
		\nonumber\\&\quad
		-
		\sum_{k=1}^\infty
		\sum_{n=0}^{[k/3]}
		(-1)^{n+k}(-z^{3n-k+1}-z^{k-3n})q^{\frac{k^2+k}{2}-3n^2-n}
		\nonumber\\&\quad
		+
		\sum_{k=1}^\infty
		\sum_{n=1}^{[k/3]}
		(-1)^{n+k}(-z^{3n-k}-z^{k-3n+1})q^{\frac{k^2-k}{2}-3n^2+n}
		\nonumber\\&\quad
		-
		\sum_{k=1}^\infty
		\sum_{n=1}^{[k/3]}
		(-1)^{n+k}(-z^{3n-k}-z^{k-3n+1})q^{\frac{k^2+k}{2}-3n^2-n}
	\\
	&=
		\sum_{n=0}^\infty
		\sum_{k=1}^\infty
		(-1)^{k}(-z^{1-k}-z^{k})q^{\frac{k^2-k+3n^2-n}{2}+3kn}(1-q^{k+n})(1-q^{2k-1})
		\\&\quad
		-
		(1+z)\sum_{n=0}^\infty q^{\frac{3n^2-n}{2}}(1-q^n)
.
\end{align*}
By 
(\ref{FirstC1Series})
we have that
\begin{align*}
	&-(1+z)\aqprod{q}{q}{\infty}^2\aqprod{q}{q^2}{\infty}
	\\
	&=
		(1+z)\aqprod{z,z^{-1}}{q}{\infty}S_{C1}(z,q)
		\\&\quad
		-
		\sum_{j=1}^\infty\sum_{n=0}^\infty (z^j + z^{1-j})
			(-1)^{j+1} q^{\frac{j(j-1)}{2} + \frac{n(3n-1)}{2} + 3nj}(1-q^{j+n})(1-q^{2j-1})
	\\
	&=
		(1+z)\aqprod{z,z^{-1}}{q}{\infty}S_{C1}(z,q)
		\\&\quad
		-
		\sum_{n=0}^\infty
		\sum_{k=1}^\infty
		(-1)^{k}(-z^{1-k}-z^{k})q^{\frac{k^2-k+3n^2-n}{2}+3kn}(1-q^{k+n})(1-q^{2k-1})
		\\&\quad
		-
		(1+z)\sum_{n=0}^\infty q^{\frac{3n^2-n}{2}}(1-q^n)
	\\
	&=
		\sum_{k=1}^\infty
		\sum_{n=0}^{[k/3]}
		(-1)^{n+k}(1+z)q^{\frac{k^2-k}{2}-3n^2+n}
		-
		\sum_{k=1}^\infty
		\sum_{n=0}^{[k/3]}
		(-1)^{n+k}(1+z)q^{\frac{k^2+k}{2}-3n^2-n}
		\nonumber\\&\quad
		+
		\sum_{k=1}^\infty
		\sum_{n=1}^{[k/3]}
		(-1)^{n+k}(1+z)q^{\frac{k^2-k}{2}-3n^2+n}
		-
		\sum_{k=1}^\infty
		\sum_{n=1}^{[k/3]}
		(-1)^{n+k}(1+z)q^{\frac{k^2+k}{2}-3n^2-n}
		\nonumber\\&\quad
		-
		(1+z)\sum_{n=0}^\infty q^{\frac{3n^2-n}{2}}(1-q^n)
,
\end{align*}
so that
\begin{align*}
	\aqprod{q}{q}{\infty}^2\aqprod{q}{q^2}{\infty}
	&=
		-\sum_{k=1}^\infty
		\sum_{n=0}^{[k/3]}
		(-1)^{n+k}q^{\frac{k^2-k}{2}-3n^2+n}
		+
		\sum_{k=1}^\infty
		\sum_{n=0}^{[k/3]}
		(-1)^{n+k}q^{\frac{k^2+k}{2}-3n^2-n}
		\nonumber\\&\quad
		-
		\sum_{k=1}^\infty
		\sum_{n=1}^{[k/3]}
		(-1)^{n+k}q^{\frac{k^2-k}{2}-3n^2+n}
		+
		\sum_{k=1}^\infty
		\sum_{n=1}^{[k/3]}
		(-1)^{n+k}q^{\frac{k^2+k}{2}-3n^2-n}
		+
		\sum_{n=0}^\infty q^{\frac{3n^2-n}{2}}(1-q^n)
	\\
	&=	
		\sum_{k=1}^\infty
		\sum_{n=0}^{[k/3]}
		(-1)^{n+k+1}q^{\frac{k^2-k}{2}-3n^2+n}
		+
		\sum_{k=1}^\infty
		\sum_{n=-[k/3]}^{0}
		(-1)^{n+k}q^{\frac{k^2+k}{2}-3n^2+n}
		\nonumber\\&\quad
		-
		\sum_{k=1}^\infty
		\sum_{n=1}^{[k/3]}
		(-1)^{n+k}q^{\frac{k^2-k}{2}-3n^2+n}
		+
		\sum_{k=1}^\infty
		\sum_{n=-[k/3]}^{-1}
		(-1)^{n+k}q^{\frac{k^2+k}{2}-3n^2+n}
		\\&\quad
		+
		\sum_{n=0}^\infty q^{\frac{3n^2-n}{2}}(1-q^n)
	\\
\end{align*}
We use that
\begin{align*}
	\sum_{k=1}^\infty
	\sum_{n=-[k/3]}^{-1}
	(-1)^{n+k}q^{\frac{k^2+k}{2}-3n^2+n}
	&=
		\sum_{k=1}^\infty
		\sum_{n=-[(k-1)/3]}^{-1}
		(-1)^{n+k+1}q^{\frac{k^2-k}{2}-3n^2+n}
	,\\
	-\sum_{k=1}^\infty\sum_{n=1}^{[k/3]}
		(-1)^{n+k}q^{\frac{k^2-k}{2}-3n^2+n}
	&=
		\sum_{k=1}^\infty\sum_{n=1}^{[(k+1)/3]}
		(-1)^{n+k}q^{\frac{k^2+k}{2}-3n^2+n}
,
\end{align*}
to find that
\begin{align*}
	\aqprod{q}{q}{\infty}^2\aqprod{q}{q^2}{\infty}
	&=
	\sum_{k=1}^\infty\sum_{n=-[(k-1)/3]}^{[k/3]}
		(-1)^{n+k+1}q^{\frac{k^2-k}{2}-3n^2+n}
	+
	\sum_{k=1}^\infty\sum_{n=-[k/3]}^{[(k+1)/3]}
		(-1)^{n+k}q^{\frac{k^2+k}{2}-3n^2+n}
	\\&\quad
	+
	\sum_{n=0}^\infty q^{\frac{3n^2-n}{2}}(1-q^n)
.
\end{align*}
But $[(k+1)/3] = [k/3]$ when $k\not\equiv 2\pmod{3}$ and
$[k/3]= -[(k-1)/3]$ when $k\not\equiv 3\pmod{3}$, so we find
\begin{align*}
	\sum_{k=1}^\infty\sum_{n=-[k/3]}^{[(k+1)/3]}
		(-1)^{n+k}q^{\frac{k^2+k}{2}-3n^2+n}
	&=
	\sum_{k=1}^\infty\sum_{n=-[(k-1)/3]}^{[k/3]}
		(-1)^{n+k}q^{\frac{k^2+k}{2}-3n^2+n}
	-
	\sum_{k=1}^\infty q^{ \frac{3k^2-k}{2} }
	+
	\sum_{k=1}^\infty q^{ \frac{3k^2+k}{2} }
.
\end{align*}
Thus
\begin{align*}
	\aqprod{q}{q}{\infty}^2\aqprod{q}{q^2}{\infty}
	&=
	\sum_{k=1}^\infty\sum_{n=-[(k-1)/3]}^{[k/3]}
		(-1)^{n+k+1}q^{\frac{k^2-k}{2}-3n^2+n}(1-q^k)
.
\end{align*}

\end{proof}

\section{Proofs of Congruences by Theorem \ref{TheoremFinalSeriesIdentities}  }

We recall to prove the congruences in Theorem
\ref{TheoremCongruences}
we are to show  
the following terms are zero:
$q^{3n}$ in $S_{A1}(\zeta_3,q)$,
$q^{3n+1}$ in $S_{A3}(\zeta_3,q)$,
$q^{5n+4}$ in $S_{A5}(\zeta_5,q)$,
$q^{7n+1}$ in $S_{A5}(\zeta_7,q)$,
$q^{5n+4}$ in $S_{A7}(\zeta_5,q)$,
$q^{5n+3}$ in $S_{C5}(\zeta_5,q)$,
$q^{3n}$ in $S_{E2}(\zeta_3,q)$,
and $q^{3n+1}$ in $S_{E4}(\zeta_3,q)$.
The double series do not appear to easily give that 
the terms $q^{5n+1}$ in $S_{A3}(\zeta_5,q)$ and
$q^{5n+3}$ in $S_{C1}(\zeta_5,q)$ are zero.

\begin{corollary}
For $n\ge 0$, 
\begin{align*}
	M_{A1}(0,3,3n)&=M_{A1}(1,3,3n)=M_{A1}(2,3,3n)=\frac{1}{3}\sptA{A1}{3n}
	,\\
	M_{A3}(0,3,3n+1)&=M_{A3}(1,3,3n+1)=M_{A3}(2,3,3n+1)=\frac{1}{3}\sptA{A3}{3n+1}
	,\\
	M_{E2}(0,3,3n)&=M_{E2}(1,3,3n)=M_{E2}(2,3,3n)=\frac{1}{3}\sptA{E2}{3n}
	,\\
	M_{E4}(0,3,3n+1)&=M_{E4}(1,3,3n+1)=M_{E4}(2,3,3n+1)=\frac{1}{3}\sptA{E4}{3n+1}
	.
\end{align*}
\end{corollary}
\begin{proof}
We are to show that $[q^{3n}]S_{A1}(\zeta_3,q) = 0$ for $n\ge 0$.
We note that
\begin{align*}
	\frac{1}{\aqprod{\zeta_3q,\zeta_3^{-1}q}{q}{\infty}}
	&=
	\frac{\aqprod{q}{q}{\infty}}{\aqprod{q^3}{q^3}{\infty}}
.
\end{align*}
By (\ref{FinalA1Series}) we have that
\begin{align*}
	&(1+\zeta_3)(1-\zeta_3)(1-\zeta_3^2)S_{A1}(\zeta_3,q)
	\\
	&=
		\frac{1}{\aqprod{q^3}{q^3}{\infty}}
		\sum_{k=1}^\infty (1-\zeta_3^{k-1})(1-\zeta_3^k)\zeta_3^{1-k}
		(-1)^{k+1}q^{k(k-1)/2} 
		\\&\quad
		\frac{1}{\aqprod{q^3}{q^3}{\infty}}
		\sum_{k=1}^\infty (1-\zeta_3^{k-1})(1-\zeta_3^k)\zeta_3^{1-k}
		\sum_{n=1}^\infty (-1)^{n+k+1} q^{\frac{k(k-1)}{2} + \frac{n(n-3)}{2} + 2kn}(1+q^n)
.
\end{align*}

Upon inspection we find that when the $q^{3N}$ terms occur, we have
either $k\equiv 0 \pmod{3}$ or $k\equiv 1 \pmod{3}$. However for such values of $k$ we have
either $(1-\zeta_3^{k})=0$ or $(1-\zeta_3^{k-1})=0$. 
Thus the coefficient of $q^{3N}$ in $S_{A1}(\zeta_3,q)$ must be zero.
The identities for $M_{A3}$, $M_{E2}$, $M_{E4}$
follow in the same fashion.
\end{proof}

\begin{corollary}
For $n\ge 0$, 
\begin{align*}
	M_{A5}(0,5,5n+4)&=M_{A5}(1,5,5n+4)=M_{A5}(2,5,5n+4)=M_{A5}(3,5,5n+4)
		\\&		
		=M_{A5}(4,5,5n+4)=\frac{1}{5}\sptA{A5}{5n+4}
	,\\
	M_{A7}(0,5,5n+4)&=M_{A7}(1,5,5n+4)=M_{A7}(2,5,5n+4)=M_{A7}(3,5,5n+4)
		\\&=M_{A7}(4,5,5n+4)=\frac{1}{5}\sptA{A7}{5n+4}
	,\\
	M_{C5}(0,5,5n+3)&=M_{C5}(1,5,5n+3)=M_{C5}(2,5,5n+3)=M_{C5}(3,5,5n+3)
		\\&=M_{C5}(4,5,5n+3)=\frac{1}{5}\sptA{C5}{5n+3}
.
\end{align*}
\end{corollary}
\begin{proof}
We are to show for $N\ge 0$ that $[q^{5N+4}]S_{A5}(\zeta_5,q) = 0$.
By Lemma 3.9 of \cite{Garvan1} we have
\begin{align*}
	\frac{1}{\aqprod{\zeta_5q,\zeta_5^{-1}q}{q}{\infty}}
	&=
	\frac{1}{\jacprod{q^5}{q^{25}}}
	+q\frac{(\zeta_5+\zeta_5^{-1})}{\jacprod{q^{10}}{q^{25}}}
.
\end{align*}

By (\ref{A5Series}) we then have
\begin{align*}
	&(1+\zeta_5)(1-\zeta_5)(1-\zeta_5^{-1})S_{A5}(\zeta_5,q)
	\\
	&=
	\frac{1}{\jacprod{q^5}{q^{25}}}
	\sum_{k=-\infty}^\infty (-1)^{k} (1-\zeta_5^k)(1-\zeta_5^{k+1})\zeta_5^{-k}q^{\frac{k(3k+1)}{2}}	
	\\&\quad
	+
	\frac{(\zeta_5+\zeta_5^{-2})q}{\jacprod{q^{10}}{q^{25}}}
	\sum_{k=-\infty}^\infty (-1)^{k} (1-\zeta_5^k)(1-\zeta_5^{k+1})\zeta_5^{-k}q^{\frac{k(3k+1)}{2}}	
.
\end{align*}

However we find that $\frac{k(3k+1)}{2}$ is never congruent to $3$ or $4$ 
modulo $5$, so we see the coefficient of $q^{5N+4}$ in $S_{A5}(\zeta_5,q)$ must
be zero.
The identities for
$M_{A7}$ and $M_{C5}$
follow similarly.

\end{proof}

\begin{corollary}
For $n\ge 0$, 
$M_{A5}(0,7,7n+1)=M_{A5}(1,7,7n+1)=M_{A5}(2,7,7n+1)=M_{A5}(3,7,7n+1)
=M_{A5}(4,7,7n+1)=M_{A5}(5,7,7n+1)=M_{A5}(6,7,7n+1)=\frac{1}{7}\sptA{A5}{7n+1}$.
\end{corollary}
\begin{proof}
We are to show for $N\ge 0$ that $[q^{7N+1}]S_{A5}(\zeta_7,q) = 0$.
To do this we actually find the $7$-dissection of $S_{A5}(\zeta_7,q)$ and
see there are no $q^{7N+1}$ terms.
We claim that
\begin{align}\label{Id1For7n1Corollary}
	S_{A5}(\zeta_7,q)
	&= \aqprod{q^{49}}{q^{49}}{\infty}
	\left(
		-(1+\zeta+\zeta^6)\frac{q^{14}}{\jacprod{q^{42},q^{49},q^{56}}{q^{147}}}	
		+q^2\frac{\jacprod{q^{14}}{q^{49}}}{\jacprod{q^7,q^{21}}{q^{49}}}
		\right.\nonumber\\&\left.\quad
		-(1+\zeta^2+\zeta^5)\frac{q^9}{\jacprod{q^{21},q^{49},q^{70}}{q^{147}}}	
		+(\zeta_7+\zeta_7^6)\frac{q^3}{\jacprod{q^{14}}{q^{49}}}
		\right.\nonumber\\&\left.\quad
		+(1+\zeta_7+\zeta_7^2+\zeta_7^5+\zeta_7^6)\frac{q^4}{\jacprod{q^{21}}{q^{49}}}
		+(\zeta+\zeta^6)q^5\frac{\jacprod{q^{35}}{q^{147}}}{\jacprod{q^{21},q^{28},q^{49},q^{49}}{q^{147}}}
		\right.\nonumber\\&\left.\quad
		+(\zeta+\zeta^6)q^{19}\frac{\jacprod{q^{14}}{q^{147}}}{\jacprod{q^{21},q^{49},q^{49},q^{70}}{q^{147}}}			
		+(2+\zeta_7^2+\zeta_7^5)\frac{q^6}{\jacprod{q^{14},q^{49},q^{63}}{q^{147}}}	
	\right)
,
\end{align}
which we see has no terms of the form $q^{7N+1}$. To begin we have by 
(\ref{A5Series}) that
\begin{align*}
	S_{A5}(\zeta_7,q)
	&=
		\frac{1}{\aqprod{\zeta_7q,\zeta_7^{-1}q}{q}{\infty}}
		\sum_{n=-\infty}^\infty	
			\frac{(-1)^n(1-\zeta_7^n)(1-\zeta_7^{n+1})}{(1+\zeta_7)(1-\zeta_7)(1-\zeta_7^{-1})}
		\zeta_7^{-n}q^{\frac{n(3n+1)}{2}}
,
\end{align*}
so we can instead verify
\begin{align}\label{Id2For7n1Corollary}
	&\frac{\aqprod{q}{q}{\infty}}{\aqprod{\zeta_7q,\zeta_7^{-1}q}{q}{\infty}}
	\sum_{n=-\infty}^\infty	
		\frac{(-1)^n(1-\zeta_7^n)(1-\zeta_7^{n+1})}{(1+\zeta_7)(1-\zeta_7)(1-\zeta_7^{-1})}
		\zeta_7^{-n}q^{\frac{n(3n+1)}{2}}
	\nonumber\\
	&= \aqprod{q}{q}{\infty}\aqprod{q^{49}}{q^{49}}{\infty}
	\left(
		-(1+\zeta+\zeta^6)\frac{q^{14}}{\jacprod{q^{42},q^{49},q^{56}}{q^{147}}}	
		+q^2\frac{\jacprod{q^{14}}{q^{49}}}{\jacprod{q^7,q^{21}}{q^{49}}}
		\right.\nonumber\\&\left.\quad
		-(1+\zeta^2+\zeta^5)\frac{q^9}{\jacprod{q^{21},q^{49},q^{70}}{q^{147}}}	
		+(\zeta_7+\zeta_7^6)\frac{q^3}{\jacprod{q^{14}}{q^{49}}}
		\right.\nonumber\\&\left.\quad
		+(1+\zeta_7+\zeta_7^2+\zeta_7^5+\zeta_7^6)\frac{q^4}{\jacprod{q^{21}}{q^{49}}}
		+(\zeta+\zeta^6)q^5\frac{\jacprod{q^{35}}{q^{147}}}{\jacprod{q^{21},q^{28},q^{49},q^{49}}{q^{147}}}
		\right.\nonumber\\&\left.\quad
		+(\zeta+\zeta^6)q^{19}\frac{\jacprod{q^{14}}{q^{147}}}{\jacprod{q^{21},q^{49},q^{49},q^{70}}{q^{147}}}			
		+(2+\zeta_7^2+\zeta_7^5)\frac{q^6}{\jacprod{q^{14},q^{49},q^{63}}{q^{147}}}	
	\right)
.
\end{align}

For the series we have
\begin{align}\label{EqA5Zeta7SeriesTerm}
	&\sum_{n=-\infty}^\infty	
		\frac{(-1)^n(1-\zeta_7^n)(1-\zeta_7^{n+1})}{(1+\zeta_7)(1-\zeta_7)(1-\zeta_7^{-1})}
		\zeta_7^{-n}q^{\frac{n(3n+1)}{2}}
	\nonumber\\
	&=
	\sum_{k=-3}^3\sum_{n=-\infty}^\infty
		\frac{(-1)^{n+k}(1-\zeta_7^k)(1-\zeta_7^{k+1})\zeta_7^{-k}}{(1+\zeta_7)(1-\zeta_7)(1-\zeta_7^{-1})}
		q^{\frac{(7n+k)(3(7n+k)+1)}{2}}
	\nonumber\\
	&=
	\sum_{k=-3}^3
		\frac{(-1)^{k}(1-\zeta_7^k)(1-\zeta_7^{k+1})\zeta_7^{-k}}{(1+\zeta_7)(1-\zeta_7)(1-\zeta_7^{-1})}
		q^{\frac{k(3k+1)}{2}}
	\sum_{n=-\infty}^\infty
		(-1)^{n}q^{21nk+77}q^{\frac{147n(n-1)}{2}}
	\nonumber\\
	&=
		(1+\zeta_7+\zeta_7^6)q^{12}\jacprod{q^{14}}{q^{147}}
			\aqprod{q^{147}}{q^{147}}{\infty}
		-q^{5}\jacprod{q^{35}}{q^{147}}
			\aqprod{q^{147}}{q^{147}}{\infty}
		\nonumber\\&\quad
		+q^{2}\aqprod{q^{49}}{q^{49}}{\infty}
		-(1+\zeta_7+\zeta_7^6)q^{7}\jacprod{q^{119}}{q^{147}}
			\aqprod{q^{147}}{q^{147}}{\infty}
		\nonumber\\&\quad
		+(1-\zeta_7^3-\zeta_7^4)q^{15}\jacprod{q^{140}}{q^{147}}
			\aqprod{q^{147}}{q^{147}}{\infty}
.
\end{align}
The last line of the above equality follows from the Jacobi Triple Product Identity.

By Theorem 5.1 of \cite{Garvan1} we have
\begin{align}\label{EqA5Zeta7CrankTerm}
	&\frac{\aqprod{q}{q}{\infty}}
		{\aqprod{\zeta_7q,\zeta_7^{-1}q}{q}{\infty}}
	\\
	&=
	\aqprod{q^{49}}{q^{49}}{\infty}
	\left(
		\frac{\jacprod{q^{21}}{q^{49}}}{\jacprod{q^{7},q^{14}}{q^{49}}}
		+(\zeta_7+\zeta_7^6-1)q\frac{1}{\jacprod{q^{7}}{q^{49}}}
		+(\zeta_7^2+\zeta_7^5)q^2\frac{\jacprod{q^{14}}{q^{49}}}{\jacprod{q^7,q^{21}}{q^{49}}}
		\right.\nonumber\\&\left.\quad
		+(\zeta_7^3+\zeta_7^4+1)q^3\frac{1}{\jacprod{q^{14}}{q^{49}}}
		-(\zeta_7+\zeta_7^6)q^4\frac{1}{\jacprod{q^{21}}{q^{49}}}
		-(\zeta_7^2+\zeta_7^5+1)q^6\frac{\jacprod{q^{7}}{q^{49}}}{\jacprod{q^{14},q^{21}}{q^{49}}}
	\right)
.
\end{align}

By Euler's Pentagonal Numbers Theorem and the Jacobi Triple Product Identity we have
\begin{align}\label{Id3For7n1Corollary}
	\aqprod{q}{q}{\infty}
	&=
	\aqprod{q^{49}}{q^{49}}{\infty}
	\left(
		\frac{\jacprod{q^{14}}{q^{49}}}{\jacprod{q^{7}}{q^{49}}}
		- q\frac{\jacprod{q^{21}}{q^{49}}}{\jacprod{q^{14}}{q^{49}}}
		- q^2
		+ q^5\frac{\jacprod{q^{7}}{q^{49}}}{\jacprod{q^{21}}{q^{49}}}
	\right)
.
\end{align}

To verify (\ref{Id2For7n1Corollary}), we multiply 
the right hand side of (\ref{Id1For7n1Corollary}) by (\ref{Id3For7n1Corollary})
collect the $q^{7N+k}$ terms, for $0\le k\le 6$, and verify each of those
is equal to the corresponding term from multiplying 
(\ref{EqA5Zeta7SeriesTerm}) by (\ref{EqA5Zeta7CrankTerm}). We do not
include the full details, but as an example, in verifying the
$q^{7N}$ terms match, we find we are to prove the following identity:
\begin{align}\label{IdMess7n1}
	&	\aqprod{q^{49}}{q^{49}}{\infty}^2
		\left(
		-(1+\zeta+\zeta^6)q^{14}
		\frac{\jacprod{q^{14}}{q^{49}}}
			{\jacprod{q^{7}}{q^{49}}\jacprod{q^{42},q^{49},q^{56}}{q^{147}}}
		\right.\nonumber\\&\quad
		- 
		(2+\zeta_7^2+\zeta_7^5)
		q^7\frac{\jacprod{q^{21}}{q^{49}}}
			{\jacprod{q^{14}}{q^{49}}\jacprod{q^{14},q^{49},q^{63}}{q^{147}}}	
		- 
		(\zeta+\zeta^6)q^7\frac{\jacprod{q^{35}}{q^{147}}}
			{\jacprod{q^{21},q^{28},q^{49},q^{49}}{q^{147}}}
		\nonumber\\&\quad
		- 
		(\zeta+\zeta^6)q^{21}\frac{\jacprod{q^{14}}{q^{147}}}
			{\jacprod{q^{21},q^{49},q^{49},q^{70}}{q^{147}}}			
		+
		q^7\frac{\jacprod{q^{7},q^{14}}{q^{49}}}{\jacprod{q^7,q^{21},q^{21}}{q^{49}}}
		\nonumber\\&\quad\left.
		-
		(1+\zeta^2+\zeta^5)q^{14}
		\frac{\jacprod{q^{7}}{q^{49}}}
			{\jacprod{q^{21}}{q^{49}}\jacprod{q^{21},q^{49},q^{70}}{q^{147}}}
		\right)
	\nonumber\\
	&=
	\aqprod{q^{49}}{q^{49}}{\infty}\aqprod{q^{147}}{q^{147}}{\infty}
	\left(
		-(1+\zeta_7+\zeta_7^6)q^{7}
		\frac{\jacprod{q^{21}}{q^{49}}\jacprod{q^{28}}{q^{147}}}{\jacprod{q^{7},q^{14}}{q^{49}}}
		-
		q^{14}
		\frac{\jacprod{q^{14}}{q^{49}}\jacprod{q^{14}}{q^{147}}}
			{\jacprod{q^7,q^{21}}{q^{49}}}
		\right.\nonumber\\&\quad
		-
		(\zeta_7^2+\zeta_7^5)q^7
		\frac{\jacprod{q^{14}}{q^{49}}\jacprod{q^{35}}{q^{147}}}
			{\jacprod{q^7,q^{21}}{q^{49}}}
		\left.
		-(\zeta_7^2+\zeta_7^5+2)q^{21}
		\frac{\jacprod{q^{7}}{q^{49}}\jacprod{q^{7}}{q^{147}}}
		{\jacprod{q^{14},q^{21}}{q^{49}}}
	\right)
.
\end{align}
Dividing both sides by 
$-\aqprod{q^{49}}{q^{49}}{\infty}^2q^{14}
	\frac{\jacprod{q^{14}}{q^{49}}}
	{\jacprod{q^{7}}{q^{49}}\jacprod{q^{42},q^{49},q^{56}}{q^{147}}}
$
and then replacing $q^7$ by $q$, we find (\ref{IdMess7n1}) to be equivalent to
\begin{align*}
	&	
		(1+\zeta+\zeta^6)
		+ 
			(2+\zeta_7^2+\zeta_7^5)q^{-1}
			\frac{\jacprod{q,q^{3}}{q^{7}}\jacprod{q^{6},q^{8}}{q^{21}}}
				{\jacprod{q^{2},q^2}{q^{7}}\jacprod{q^{2},q^{9}}{q^{21}}  }	
		\\&\quad
		+ 
			(\zeta+\zeta^6)q^{-1}
			\frac{\jacprod{q}{q^{7}} \jacprod{q^{5},q^{6},q^{8}}{q^{21}}}
				{\jacprod{q^{2}}{q^{7}} \jacprod{q^{3},q^{4},q^{7}}{q^{21}}   }
		+ 
			(\zeta+\zeta^6)q
			\frac{\jacprod{q}{q^{7}} \jacprod{q^2, q^{6},q^{8}}{q^{21}}}
			{\jacprod{q^{2}}{q^{7}} \jacprod{q^{3},q^{7},q^{10}}{q^{21}}  }			
		\\&\quad
		-
			q^{-1}
			\frac{\jacprod{q,q}{q^{7}} \jacprod{q^{6},q^{7},q^{8}}{q^{21}}}
			{\jacprod{q,q^{3},q^{3}}{q^{7}}  }
		+
			(1+\zeta^2+\zeta^5)
			\frac{\jacprod{q,q}{q^{7}} \jacprod{q^{6},q^{8}}{q^{21}}}
				{\jacprod{q^2,q^3}{q^{7}}\jacprod{q^{3},q^{10}}{q^{21}} }
	\\
	&=
			(1+\zeta_7+\zeta_7^6)q^{-1}
			\frac{\jacprod{q,q^{3}}{q^{7}} \jacprod{q^4,q^{6},q^{8}}{q^{21}}}
				{\jacprod{q,q^{2},q^2}{q^{7}}}
		+
			\frac{\jacprod{q}{q^{7}} \jacprod{q^2,q^{6},q^{8}}{q^{21}}}
				{\jacprod{q,q^{3}}{q^{7}}   }
		\\&\quad
		+
			(\zeta_7^2+\zeta_7^5)q^{-1}
			\frac{\jacprod{q}{q^{7}}\jacprod{q^5,q^{6},q^{8}}{q^{21}}}
				{\jacprod{q,q^{3}}{q^{7}} }
		+
			(\zeta_7^2+\zeta_7^5+2)q
			\frac{\jacprod{q,q}{q^{7}} \jacprod{q,q^{6},q^{8}}{q^{21}}}
				{\jacprod{q^2,q^{2},q^{3}}{q^{7}} }
.
\end{align*}
In this form each of the ten terms is a modular function with respect to
$\Gamma_1(21)$ by Theorem 3 of \cite{Robins}. By moving all ten terms to
one side of the equation, we are to verify a modular function is identically
zero. We do this by checking that otherwise this modular function would violate
the valence formula. 
Using Theorem 4 of \cite{Robins} we compute the order at each cusp of 
$\Gamma_1(21)$, not equivalent to infinity,
of each generalized eta quotient. At each cusp we taking the minimal order at 
that cusp among the generalized eta quotients. Summing these minimal orders 
gives $-13$. By the valence formula, if the order at $\infty$ is $14$ or larger, 
the sum of generalized eta quotients is identically zero. This is
easily verified with Maple. 

We then repeat this process for the other six values of $k$. In each case we have
to prove an identity between several infinite products. Dividing by one of the terms yields
an identity between modular functions with respect to $\Gamma_1(21)$. We examine 
the orders of each modular function at the cusps to and find we must show the order
at $\infty$ is larger than some number.
In all cases this number is rather small, with $13$ being the largest.

\end{proof}

\section{Dissection Formulas}

We recall we are to find the $3$-dissections of 
$S_{E2}(\zeta_3,q)$ and $S_{E4}(\zeta_3,q)$, and the $5$-dissections of 
$S_{C1}(\zeta_5,q)$ and  $S_{C5}(\zeta_5,q)$.
We prove Theorem \ref{TheoremAllDissections} by relating
each series $S_X(z,q)$ to an appropriate rank-like and crank-like function, both
of which have dissections that are either known or easy to deduce.

We recall $R(z,q)$ is the generating function for the rank of ordinary 
partitions. By \cite{Watson}, one form of the generating function of the 
rank is
\begin{align}
	R(z,q) &=
	\sum_{n=0}^\infty\sum_{m=-\infty}^\infty N(m,n)z^mq^n
	=
	\frac{1}{\aqprod{q}{q}{\infty}}\left(
		1+\sum_{n=1}^\infty \frac{(1-z)(1-z^{-1})(-1)^nq^{\frac{n(3n+1)}{2}}(1+q^n)}
		{(1-zq^n)(1-z^{-1}q^n)}
	\right)
.
\end{align}
We recall $C(z,q)$ is the generating function for the crank of ordinary 
partitions. By \cite{Garvan1}, two forms of the generating function 
are given by
\begin{align*}
	C(z,q)
	&=
	\frac{\aqprod{q}{q}{\infty}}{\aqprod{zq,z^{-1}q}{q}{\infty}}
	=
	\frac{1}{\aqprod{q}{q}{\infty}}\left(
		1+\sum_{n=1}^\infty \frac{(1-z)(1-z^{-1})(-1)^nq^{\frac{n(n+1)}{2}}(1+q^n)}
		{(1-zq^n)(1-z^{-1}q^n)}
	\right)
.
\end{align*}
We also define two series of a similar form to that of $R(z,q)$ and $C(z,q)$,
\begin{align*}
	R_1(z,q) 
	&= 
	\frac{\aqprod{-q}{q}{\infty}}{\aqprod{q}{q}{\infty}}
	\left(
		1+2\sum_{n=1}^\infty \frac{(1-z)(1-z^{-1})(-1)^nq^n}{(1-zq^n)(1-z^{-1}q^n)}
	\right)
	,\\
	R_2(z,q) 
	&= 
	\frac{\aqprod{-q}{q}{\infty}}{\aqprod{q}{q}{\infty}}
	\left(
		1+\sum_{n=1}^\infty \frac{(1-z)(1-z^{-1})(-1)^nq^{n^2}(1+q^{2n})}{(1-zq^n)(1-z^{-1}q^n)}
	\right)
.
\end{align*}
We used $R_1(z,q)$ in \cite{GarvanJennings} in proving dissections similar
to what we are doing here, and as we will see shortly $R_2(z,q)$ is essentially
the rank of an overpartition.

\begin{proposition}\label{TheoremRankCrankDifferences}
\begin{align}
	\label{C1RankCrankDifference}
	S_{C1}(z,q)	
	&=
	\frac{1}{(1-z)(1-z^{-1})}\left( 
		R(z,q^2) - 
		\aqprod{q}{q^2}{\infty}C(z,q)
	\right)
	,\\
	\label{C5RankCrankDifference}
	S_{C5}(z,q)	
	&=
	\frac{1}{(1-z)(1-z^{-1})}\left( 
		C(z,q^2) - 
		\aqprod{q}{q^2}{\infty}C(z,q)
	\right)
	,\\
	\label{E2RankCrankDifference}
	S_{E2}(z,q)	
	&=
	\frac{1}{(1-z)(1-z^{-1})}\left( 
		R_1(z,q) - 
		\aqprod{-q}{q}{\infty}C(z,q)
	\right)
	,\\
	\label{E4RankCrankDifference}
	S_{E4}(z,q)	
	&=
	\frac{1}{(1-z)(1-z^{-1})}\left( 
		R_2(z,q) - 
		\aqprod{-q}{q}{\infty}C(z,q)
	\right)
.
\end{align}
\end{proposition}

We note that $\aqprod{-q}{q}{\infty}C(z,q) 
= \frac{\aqprod{-q,q}{q}{\infty}}{\aqprod{zq,z^{-1}q}{q}{\infty}}$ is a 
residual crank studied in \cite{BLO2} and \cite{GarvanJennings}.
Most of these functions have known dissections when $z$ is either a primitive 
third or fifth root of unity.

By taking subtracting the expressions for $S_{C1}(z,q)$ and
$S_{C5}(z,q)$ in Proposition \ref{TheoremRankCrankDifferences}
and using (\ref{EqIntro1}),
we have Corollary \ref{CorollaryC1AndC5Identities}.

\begin{proposition}\label{PropositionRankCrankDissections}
\begin{align}
	\label{EquationRank5Dissection}
	&R(\zeta_5,q)
	=
		\frac{\aqprod{q^{25}}{q^{25}}{\infty}\jacprod{q^{10}}{q^{25}}}
			{\jacprod{q^{5}}{q^{25}}^2}
		+
		q^5\frac{\zeta_5+\zeta_5^{-1}-2}{\aqprod{q^{25}}{q^{25}}{\infty}}
			\sum_{n=-\infty}^\infty \frac{(-1)^nq^{75n(n+1)/2}}{1-q^{25n+5}}
		+
		q\frac{\aqprod{q^{25}}{q^{25}}{\infty}}{\jacprod{q^5}{q^{25}}}
		\nonumber\\&\quad
		+
		q^2(\zeta_5+\zeta_5^{-1})\frac{\aqprod{q^{25}}{q^{25}}{\infty}}
			{\jacprod{q^{10}}{q^{25}}}
		-
		q^3(\zeta_5+\zeta_5^{-1})\frac{\aqprod{q^{25}}{q^{25}}{\infty}\jacprod{q^5}{q^{25}}}
			{\jacprod{q^{10}}{q^{25}}^2}
		\nonumber\\&\quad
		-
		q^8\frac{2\zeta_5+2\zeta_5^{-1}+1}{\aqprod{q^{25}}{q^{25}}{\infty}}
			\sum_{n=-\infty}^\infty \frac{(-1)^nq^{75n(n+1)/2}}{1-q^{25n+10}}
	,\\
	\label{EquationCrank5Dissection}
	&C(\zeta_5,q)	
	=
		\aqprod{q^{25}}{q^{25}}{\infty}
		\left(
			\frac{\jacprod{q^{10}}{q^{25}}}{\jacprod{q^{5}}{q^{25}}^2}
			+(\zeta_5+\zeta_5^4-1)q\frac{1}{\jacprod{q^{5}}{q^{25}}}
			-(\zeta_5+\zeta_5^4+1)q^2\frac{1}{\jacprod{q^{10}}{q^{25}}}
			\right.\nonumber\\&\left.\quad
			-(\zeta_5+\zeta_5^4)q^3\frac{\jacprod{q^{5}}{q^{25}}}{\jacprod{q^{10}}{q^{25}}^2}
		\right)
	,\\
	\label{ThreeDissectionExtraSeries}
	&R_1(\zeta_3,q)
	= 
		\frac{\aqprod{q^9}{q^9}{\infty}^4\aqprod{q^6}{q^6}{\infty}}
			{\aqprod{q^3}{q^3}{\infty}^2\aqprod{q^{18}}{q^{18}}{\infty}^2}
		-
		4q\frac{\aqprod{q^{18}}{q^{18}}{\infty}\aqprod{q^9}{q^9}{\infty}}
			{\aqprod{q^3}{q^3}{\infty}}
		+
		q^2\frac{\aqprod{q^{18}}{q^{18}}{\infty}^4}
			{\aqprod{q^9}{q^9}{\infty}^2\aqprod{q^6}{q^6}{\infty}}
	,\\
	\label{TheoremExtraRank3Dissection}
	&R_2(\zeta_3,q) 
	=
		\frac{3}{2}
		-\frac{\aqprod{q^9}{q^9}{\infty}^4\aqprod{q^6}{q^6}{\infty}}
			{2\aqprod{q^{18}}{q^{18}}{\infty}^2\aqprod{q^3}{q^3}{\infty}^2}
		-
		q\frac{\aqprod{q^9}{q^9}{\infty}\aqprod{q^{18}}{q^{18}}{\infty}}
			{\aqprod{q^3}{q^3}{\infty}}
		-
		2q^2\frac{\aqprod{q^{18}}{q^{18}}{\infty}^4}
			{\aqprod{q^9}{q^9}{\infty}^2\aqprod{q^6}{q^6}{\infty}}
		\nonumber\\&\quad
		+
		3q^2\frac{\aqprod{q^{18}}{q^{18}}{\infty}}{\aqprod{q^9}{q^9}{\infty}^2}
		\sum_{n=-\infty}^\infty\frac{(-1)^nq^{9n^2+9n}}{1-q^{9n+3}} 
	,\\
	\label{TheoremNewCrankDissection}
	&\aqprod{q}{q^2}{\infty}C(\zeta_5,q) 
	=
		\aqprod{q^{50}}{q^{50}}{\infty}
		\left(	
			\frac{\jacprod{q^{25}}{q^{50}}}{\jacprod{q^{5}}{q^{25}}}
			+2(\zeta_5+\zeta_5^{-1})q^5\frac{\jacprod{q^{5}}{q^{50}}}{\jacprod{q^{10}}{q^{25}}}	
			-2q\frac{\jacprod{q^{15}}{q^{50}}}{\jacprod{q^{5}}{q^{25}}}
			\right.\nonumber\\&\quad
			+(\zeta_5+\zeta_5^{-1})q\frac{\jacprod{q^{25}}{q^{50}}}{\jacprod{q^{10}}{q^{25}}}	
			-2(\zeta_5+\zeta_5^{-1})q^2\frac{\jacprod{q^{15}}{q^{50}}}{\jacprod{q^{10}}{q^{25}}}	
		\left.
			+2q^4\frac{\jacprod{q^{5}}{q^{50}}}{\jacprod{q^{5}}{q^{25}}}
		\right)
	,\\
	\label{ThreeDissectionOverpartionCrank}
	&\aqprod{-q}{q}{\infty}C(\zeta_3,q)
	=
		\frac{\aqprod{q^9}{q^9}{\infty}^4\aqprod{q^6}{q^6}{\infty}}
			{\aqprod{q^3}{q^3}{\infty}^2\aqprod{q^{18}}{q^{18}}{\infty}^2}
		-
		q\frac{\aqprod{q^{18}}{q^{18}}{\infty}\aqprod{q^9}{q^9}{\infty}}
			{\aqprod{q^3}{q^3}{\infty}}
		-
		2q^2\frac{\aqprod{q^{18}}{q^{18}}{\infty}^4}
			{\aqprod{q^9}{q^9}{\infty}^2\aqprod{q^6}{q^6}{\infty}}
.
\end{align}
\end{proposition}

We find that Theorem \ref{TheoremAllDissections} follows from 
Propositions \ref{TheoremRankCrankDifferences} and
\ref{PropositionRankCrankDissections}.

\begin{proof}[Proof of Proposition \ref{TheoremRankCrankDifferences}]

We apply Bailey's Lemma, with $\rho_1=z$ and $\rho_2=z^{-1}$,
to each $S_X(z,q)$ with the corresponding 
Bailey pair, the identities then quickly follow.
We write
\begin{align*}
	S_{C1}(z,q)
	&=
	\frac{\aqprod{q}{q^2}{\infty}\aqprod{q}{q}{\infty}}{\aqprod{z,z^{-1}}{q}{\infty}}
	\sum_{n=0}^\infty \frac{\aqprod{z,z^{-1}}{q}{n}q^n}{\aqprod{q}{q^2}{n}\aqprod{q}{q}{n}}
	-	\frac{\aqprod{q}{q^2}{\infty}\aqprod{q}{q}{\infty}}{\aqprod{z,z^{-1}}{q}{\infty}}
	\\
	&=
	\frac{\aqprod{q}{q^2}{\infty}\aqprod{q}{q}{\infty}}{\aqprod{z,z^{-1}}{q}{\infty}}
	\sum_{n=0}^\infty \aqprod{z,z^{-1}}{q}{n}q^n \beta^{C1}_n
	-	\frac{\aqprod{q}{q^2}{\infty}\aqprod{q}{q}{\infty}}{\aqprod{z,z^{-1}}{q}{\infty}}
	\\
	&=
	\frac{1}{(1-z)(1-z^{-1})\aqprod{q^2}{q^2}{\infty}}
	\sum_{n=0}^\infty \frac{(1-z)(1-z^{-1}) q^n \alpha^{C1}_n}{(1-zq^n)(1-z^{-1}q^n)}
	-	
	\frac{\aqprod{q}{q^2}{\infty}\aqprod{q}{q}{\infty}}{\aqprod{z,z^{-1}}{q}{\infty}}
	\\
	&=
	\frac{1}{(1-z)(1-z^{-1})\aqprod{q^2}{q^2}{\infty}}
	\left(
		1
		+
		\sum_{n=1}^\infty \frac{(1-z)(1-z^{-1})(-1)^n q^{3n^2 + n}(1+q^{2n})}{(1-zq^{2n})(1-z^{-1}q^{2n})}
	\right)	
		\\&\quad
	-	
	\frac{\aqprod{q}{q^2}{\infty}\aqprod{q}{q}{\infty}}{\aqprod{z,z^{-1}}{q}{\infty}}
	\\
	&=
	\frac{1}{(1-z)((1-z^{-1}))}\left( 
		R(z,q^2) - 
		\aqprod{q}{q^2}{\infty}C(z,q)
	\right)
.
\end{align*}
This proves (\ref{C1RankCrankDifference}).
The identities for $S_{C5}(z,q)$, $S_{E2}(z,q)$, and $S_{E4}(z,q)$ are proved
in the same fashion by instead using the Bailey pairs
$(\alpha^{C5},\beta^{C5})$, $(\alpha^{E2},\beta^{E2})$, and
$(\alpha^{E4},\beta^{E4})$ respectively.

\end{proof}

\begin{proof}[Proof of Proposition \ref{PropositionRankCrankDissections}]

Equation (\ref{EquationRank5Dissection}) follows from Theorem 4 of \cite{AS},
(\ref{EquationCrank5Dissection}) is (3.8) of \cite{Garvan1},
(\ref{ThreeDissectionExtraSeries}) is Theorem 2.13 of \cite{GarvanJennings},
and (\ref{ThreeDissectionOverpartionCrank}) is Theorem 2.9 of 
\cite{GarvanJennings}.

We first prove (\ref{TheoremNewCrankDissection}).
By Gauss and the Jacobi Triple Product identity we have
\begin{align}
	\aqprod{q}{q^2}{\infty}\aqprod{q}{q}{\infty}
	&=
	\sum_{n=-\infty}^\infty (-1)^n q^{n^2}
	=
	\sum_{k=-2}^2 (-1)^k q^{k^2} \sum_{n=-\infty}^\infty (-1)^n q^{10nk}q^{25n^2}
	\nonumber\\
	&=	
	\aqprod{q^{50}}{q^{50}}{\infty}\sum_{k=-2}^2 (-1)^k q^{k^2} \jacprod{q^{10k+25}}{q^{50}}
	\nonumber\\\label{ProductCrankPart1}
	&=
	\aqprod{q^{50}}{q^{50}}{\infty}
	\left(
		\jacprod{q^{25}}{q^{50}} - 2q\jacprod{q^{15}}{q^{50}}	+2q^4\jacprod{q^5}{q^{50}}	
	\right)
.
\end{align}

By Lemma 3.9 of \cite{Garvan1} we have
\begin{align}\label{ProductCrankPart2}
	\frac{1}{\aqprod{\zeta_5q,\zeta_5^{-1}q}{q}{\infty}}
	&=
	\frac{1}{\jacprod{q^{5}}{q^{25}}} 
	+ (\zeta_5+\zeta_5^{-1})q\frac{1}{\jacprod{q^{10}}{q^{25}}}
.
\end{align}

Multiplying (\ref{ProductCrankPart1}) and (\ref{ProductCrankPart2}) yields
\begin{align*}
	\frac{\aqprod{q}{q^2}{\infty}\aqprod{q}{q}{\infty}}
		{\aqprod{\zeta_5q,\zeta_5^{-1}q}{q}{\infty}}
	&=
	\aqprod{q^{50}}{q^{50}}{\infty}
	\left(	
		\frac{\jacprod{q^{25}}{q^{50}}}{\jacprod{q^{5}}{q^{25}}}
		+2(\zeta_5+\zeta_5^{-1})q^5\frac{\jacprod{q^{5}}{q^{50}}}{\jacprod{q^{10}}{q^{25}}}	
		-2q\frac{\jacprod{q^{15}}{q^{50}}}{\jacprod{q^{5}}{q^{25}}}
		\right.\\&\quad\left.
		+(\zeta_5+\zeta_5^{-1})q\frac{\jacprod{q^{25}}{q^{50}}}{\jacprod{q^{10}}{q^{25}}}	
		-2(\zeta_5+\zeta_5^{-1})q^2\frac{\jacprod{q^{15}}{q^{50}}}{\jacprod{q^{10}}{q^{25}}}	
		+2q^4\frac{\jacprod{q^{5}}{q^{50}}}{\jacprod{q^{5}}{q^{25}}}
	\right)
,
\end{align*}
which is (\ref{TheoremNewCrankDissection}).

We note we could also use the 5-dissection of R{\o}dseth \cite{Rodseth} for 
$\aqprod{-q}{q^2}{\infty}$ and replace $q$ with $-q$ to get
\begin{align*}
	\aqprod{q}{q^2}{\infty}
	&=
		\frac{\aqprod{q^5}{q^5}{\infty}\aqprod{q^{25}}{q^{25}}{\infty}^2\jacprod{q^{15}}{q^{50}}}
			{\aqprod{q^{10}}{q^{10}}{\infty}^3}
		-
		q\frac{\aqprod{q^5}{q^5}{\infty}\aqprod{q^{50}}{q^{50}}{\infty}^2\jacprod{q^{15}}{q^{50}}^2}
			{\aqprod{q^{10}}{q^{10}}{\infty}^3}
		\\&\quad
		-
		q^7\frac{\aqprod{q^5}{q^5}{\infty}\aqprod{q^{50}}{q^{50}}{\infty}^2\jacprod{q^{5}}{q^{50}}^2}
			{\aqprod{q^{10}}{q^{10}}{\infty}^3}
		-
		q^3\frac{\aqprod{q^5}{q^5}{\infty}\aqprod{q^{25}}{q^{25}}{\infty}^2\jacprod{q^{5}}{q^{50}}}
			{\aqprod{q^{10}}{q^{10}}{\infty}^3}
		\\&\quad
		+
		q^4\frac{\aqprod{q^5}{q^5}{\infty}^2\aqprod{q^{50}}{q^{50}}{\infty}^3}
			{\aqprod{q^{10}}{q^{10}}{\infty}^4\aqprod{q^{25}}{q^{25}}{\infty}}
.
\end{align*}
However multiplying this with the 5-dissection for $C(\zeta_5,q)$ would be a
much longer calculation.

For (\ref{TheoremExtraRank3Dissection})
we recall $\overline{R}(z,q)$ is the generating function for the Dyson rank of 
an overpartition. By \cite{Lovejoy1} the generating function for the 
Dyson rank of an overpartition is given by
\begin{align*}
	\overline{R}(z,q) &=
	\frac{\aqprod{-q}{q}{\infty}}{\aqprod{q}{q}{\infty}}
	\left(1+
		2\sum_{n=1}^\infty \frac{(1-z)(1-z^{-1})(-1)^n q^{n^2+n}}{(1-zq^n)(1-zq^n)}
	\right).
\end{align*}

By Theorem 1.1 of \cite{LO1} we have
\begin{align*}
	\overline{R}(\zeta_3,q)
	&=
		\frac{\aqprod{q^9}{q^9}{\infty}^4\aqprod{q^6}{q^6}{\infty}}
			{\aqprod{q^{18}}{q^{18}}{\infty}^2\aqprod{q^3}{q^3}{\infty}^2}
		+
		2q\frac{\aqprod{q^9}{q^9}{\infty}\aqprod{q^{18}}{q^{18}}{\infty}}
			{\aqprod{q^3}{q^3}{\infty}}
		+
		4q^2\frac{\aqprod{q^{18}}{q^{18}}{\infty}^4}
			{\aqprod{q^9}{q^9}{\infty}^2\aqprod{q^6}{q^6}{\infty}}
		\\&\quad
		-
		6q^2\frac{\aqprod{q^{18}}{q^{18}}{\infty}}{\aqprod{q^9}{q^9}{\infty}^2}
		\sum_{n=-\infty}^\infty\frac{(-1)^nq^{9n^2+9n}}{1-q^{9n+3}}
.
\end{align*} 

We note the expression on the right hand side of the identity in
(\ref{TheoremExtraRank3Dissection})
is $\frac{3 - \overline{R}(\zeta_3,q)}{2}$. 
We will show that
\begin{align*}
	\frac{(1-z)(1-z^{-1}) + (z+z^{-1})\overline{R}(z,q)}{2}
	&= R_2(z,q)
.
\end{align*}

We have
\begin{align*}
	&\frac{(1-z)(1-z^{-1}) + (z+z^{-1})\overline{R}(z,q)}{2}
	\\
	&=
	\frac{\aqprod{-q}{q}{\infty}}{\aqprod{q}{q}{\infty}}
	\left(
		\frac{z+z^{-1}}{2} 
		+ \frac{(1-z)(1-z^{-1})}{2}\frac{\aqprod{q}{q}{\infty}}{\aqprod{-q}{q}{\infty}}
		+ (z+z^{-1})\sum_{n=1}^\infty\frac{(1-z)(1-z^{-1})(-1)^n q^{n^2+n}}
			{(1-zq^n)(1-z^{-1}q^n)}
	\right)
	\\
	&=
	\frac{\aqprod{-q}{q}{\infty}}{\aqprod{q}{q}{\infty}}
	\left(
		1 
		+ \sum_{n=1}^\infty(1-z)(1-z^{-1})(-1)^n q^{n^2}
		+ (z+z^{-1})\sum_{n=1}^\infty\frac{(1-z)(1-z^{-1})(-1)^n q^{n^2+n}}
			{(1-zq^n)(1-z^{-1}q^n)}
	\right)
	\\
	&=
	\frac{\aqprod{-q}{q}{\infty}}{\aqprod{q}{q}{\infty}}
	\left(
		1 
		+ \sum_{n=1}^\infty (1-z)(1-z^{-1})(-1)^n q^{n^2}
			\left(
				1+\frac{(z+z^{-1})q^n}{(1-zq^n)(1-z^{-1}q^n)}
			\right)
	\right)
	\\
	&=
	\frac{\aqprod{-q}{q}{\infty}}{\aqprod{q}{q}{\infty}}
	\left(
		1 
		+ 
		\sum_{n=1}^\infty \frac{(1-z)(1-z^{-1})(-1)^n q^{n^2}(1+q^{2n})}
			{(1-zq^n)(1-z^{-1}q^n)}
	\right)
	\\
	&=
	R_2(z,q)
	.
\end{align*}
This proves (\ref{TheoremExtraRank3Dissection}), noting
$(1-\zeta_3)(1-\zeta_3^{-1}) = 3$ and $\zeta_3+\zeta_3^{-1}=-1$.
\end{proof}

We now have a second proof of the identities in Section 4 for 
$M_{C5}$, $M_{E2}$, and $M_{E4}$, as well as the only proof of the 
corresponding identity for $M_{C1}$.

\begin{corollary}
For $n\ge 0$,
\begin{align*}
	M_{C1}(0,5,5n+3)
	&=	
	M_{C1}(1,5,5n+3)
	=
	M_{C1}(2,5,5n+3)
	=
	M_{C1}(3,5,5n+3)
	=
	M_{C1}(4,5,5n+3)
	\\&
	=
	\frac{1}{5}\sptA{C1}{5n+3}
.
\end{align*}
\end{corollary}

\section{Conclusions}

We see spt-crank-type functions lead to a large number of new spt-type
functions, as well as congruences for these functions. We see some of these
spt-crank-type functions have surprising product representations, while others
have interesting Hecke-Rogers-type double series representations. Also in all 
of our cases
these functions arise as the difference of a rank-like and crank-like function.
While these results are quite different, they both come out of Bailey pairs
and Bailey's Lemma.

As noted in Section 2, here we do not pursue interpreting the $M_{X}(m,n)$ as
a statistic defined on the smallest parts that each function
$\sptA{X}{n}$ counts. However there is reason to believe such an interpretation
exists. By
\cite{JS} we know each of $M_{A1}(m,n)$, $M_{A3}(m,n)$, $M_{A5}(m,n)$, and
$M_{A7}(m,n)$ is nonnegative. Also by expanding the summands of $S_{E4}(z,q)$
by the $q$-binomial theorem,
\begin{align*}
	\frac{q^{2n}\aqprod{q^{2n+2}}{q^2}{\infty}}
		{\aqprod{zq^n,z^{-1}q^n}{q}{\infty}}
	&=
	\frac{q^{2n}\aqprod{q^{2n}}{q}{\infty}}
		{(1-q^{2n})\aqprod{q^{2n+1}}{q^2}{\infty}\aqprod{zq^n,z^{-1}q^n}{q}{\infty}}
	\\
	&=
	\frac{q^{2n}}{(1-q^{2n})\aqprod{q^{2n+1}}{q^2}{\infty}}
		\sum_{k=0}^\infty 
		\frac{z^{-k}q^{nk}}{\aqprod{q}{q}{k}\aqprod{zq^{n+k}}{q}{\infty}}
,
\end{align*}
we see that each $M_{E4}(m,n)$ is nonnegative. Numerical evidence
suggests that also each $M_{C1}(m,n)$ and $M_{C5}(m,n)$ is nonnegative. 
We pose the problem of finding nice combinatorial interpretations of the 
coefficients 
$M_{C1}(m,n)$ and $M_{C5}(m,n)$ which prove nonnegativity.

In a coming paper, the second author continues this study with Bailey pairs from groups
$B$, $F$, $G$, and $J$ of \cite{Slater} and \cite{Slater2}. This will cover the last of
the Bailey pairs of Slater that appear to give congruences of this form via
spt-crank-type functions. Of course that is not to say there are not many more
spt-crank-type functions to uncover
or that the spt-crank-type functions corresponding to the other Bailey pairs of
Slater are not also of interest.

\bibliographystyle{abbrv}
\bibliography{ranksRef}

\end{document}